\documentclass[english]{extarticle}
\usepackage[T1]{fontenc}
\usepackage[latin9]{inputenc}
\usepackage[active]{srcltx}
\usepackage{babel}
\usepackage{textcomp}
\usepackage{mathrsfs}
\usepackage{amsmath}
\usepackage{amsthm}
\usepackage{amssymb}
\usepackage[all]{xy}
\usepackage[unicode=true,pdfusetitle,
 bookmarks=true,bookmarksnumbered=false,bookmarksopen=false,
 breaklinks=false,pdfborder={0 0 1},backref=false,colorlinks=false]
 {hyperref}

\makeatletter

\providecommand{\tabularnewline}{\\}

\numberwithin{equation}{section}
\numberwithin{figure}{section}
\theoremstyle{plain}
\newtheorem{thm}{\protect\theoremname}
\theoremstyle{remark}
\newtheorem{rem}[thm]{\protect\remarkname}
\theoremstyle{plain}
\newtheorem*{thm*}{\protect\theoremname}
\theoremstyle{plain}
\newtheorem*{cor*}{\protect\corollaryname}
\theoremstyle{remark}
\newtheorem*{rem*}{\protect\remarkname}
\theoremstyle{definition}
\newtheorem{example}[thm]{\protect\examplename}
\theoremstyle{plain}
\newtheorem{conjecture}[thm]{\protect\conjecturename}
\theoremstyle{definition}
\newtheorem{defn}[thm]{\protect\definitionname}
\theoremstyle{plain}
\newtheorem{fact}[thm]{\protect\factname}
\theoremstyle{plain}
\newtheorem{lem}[thm]{\protect\lemmaname}
\theoremstyle{plain}
\newtheorem{prop}[thm]{\protect\propositionname}
\theoremstyle{plain}
\newtheorem{cor}[thm]{\protect\corollaryname}

\usepackage{tikz}

\makeatother

\providecommand{\conjecturename}{Conjecture}
\providecommand{\corollaryname}{Corollary}
\providecommand{\definitionname}{Definition}
\providecommand{\examplename}{Example}
\providecommand{\factname}{Fact}
\providecommand{\lemmaname}{Lemma}
\providecommand{\propositionname}{Proposition}
\providecommand{\remarkname}{Remark}
\providecommand{\theoremname}{Theorem}

\begin{document}
\title{Surfaces and p-adic fields I: Dehn twists}
\author{Nadav Gropper}
\maketitle
\begin{abstract}
Under the philosophy of arithmetic topology, one would like to have
an analogy between surfaces $S$, and p-adic fields $K$.

In the following we describe a point of view which helps look at surfaces
and p-adic fields in a ``uniform way'', and show that results on
mapping class groups can be extended to this point of view, and thus
be applied to $G_{K}$, the absolute Galois groups of the p-adic field
$K$.

By moving both groups to the world of pro-p groups (for $G_{K}$ we
take its maximal pro-p quotient, and for $\pi_{1}(S)$ we take its
pro-p completion), we see they both are pro-p Poincare duality groups
of dimension 2, also known as Demuskin groups. Such groups have a
very nice classification in terms of generators and relations.

Next, in order to move geometric ideas to a purely group theoretic
setting, we use the language of graphs of groups and Bass-Serre trees. 

By restating results for surfaces purely in such a manner, we can
(after proving a few technical results) get similar results for p-adic
fields.

We start by examining discrete groups with Demuskin type relations,
for which we show a few pro-p rigidity type results.

Using this we show that all splittings of a Demuskin group come from
a discrete splitting, which in turn helps us show that Dehn twists
make sense in such a context.

This gives us a family of infinite order Outer automorphisms of $G_{K}(p)$,
the maximal pro-p quotient of $G_{K}$, which are ``arithmetic Dehn
twists''. On the other hand, when specializing this to Demuskin groups
coming from surface groups, one gets back the usual definition of Dehn twists
on surfaces.

We also get a curve complex for p-adic fields, namely the complex
whose edges are graph of groups with $G_{K}(p)$ as their fundamental
group, and 2 vertices have an edge between them if the graphs corresponding
to the vertices are compatible. By the relations between discrete
and pro-p Demuskin groups, we get a nice way to understand said curve
complex. 

As a finally corollary, we show that there in an infinite family of
non-isomorphic discrete groups, having isomorphic pro-$l$ completions
for all primes $l$ (which are free pro-$l$ for $l\neq p$ and Demuskin
for $l=p$).
\end{abstract}
\newpage

$ $

\section{Introduction}

The purpose of this paper is to explain a new point of view, which
furthers the analogies of arithmetic topology, and allows for a more
``uniform'' approach towards p-adic fields $K$, and surfaces $S$.

Earlier reasons for such analogs, come from studying of the structure
of absolute Galois groups of p-adic fields, one gets a presentation
very similar to that of a surface group.

A family of groups containing both, is pro-p Poincare Duality groups
of dimension 2, which are also known as Demuskin groups. 

Another place where the similarities continue is when looking at 3-dimensional
hyperbolic manifolds, such manifold have Mostow's rigidity theorem,
which states that a complete finite volume n-dimensional hyperbolic
manifold is determined by its fundamental group. 

One has a similar theorem for number fields, which states that the
absolute Galois group of a number field determines the field (see
Neukirch, Uchida \cite{neukirch1969kennzeichnung,uchida1976isomorphisms}).

$ $

The automorphism in definition \ref{def: neuk outer}, defined by
Neukirch-Schmidt-Wingberg \cite{neukirch2013cohomology} is an important
example in anabelian geometry, showing that there is no direct analogue
for results such as Neukirch, Uchida \cite{uchida1976isomorphisms,neukirch1969kennzeichnung}
and is the only constructed outer automorphism of $G_{K}$, the absolute
Galois group of $K$, the author could find.

The argument given in \cite{neukirch2013cohomology} only shows that
the automorphism is an outer automorphism for $G_{K}$, but it does
not work for powers of that automorphism.

If one looks at representations which are trivial mod p, we get that
the representation factors through $G_{K}(p)$, the maximal pro-p
quotient of $G_{K}$ (i.e the Galois group of the maximal p-extension).

One can define at a similar construction to \ref{def: neuk outer},
which gives an automorphism of $G_{K}(p)$, but again it is not immediate
it will be outer.
\begin{rem}
One of our interests in $Out(G_{K})$ is its action on deformation
spaces of Galois representations, if one looks at representations
which are trivial mod p, we get that the representation factors through
the maximal pro-p quotient of $G_{K}$, $G_{K}(p)$.
\end{rem}

Going back to surfaces, one has their mapping class group, which can
also be seen as $Out(\pi_{1}(S))$ for a surface $S$ (see fact \ref{fact:Dehn-Nilsen}).

The theory of mapping class groups of surfaces is extremely rich,
and has given many interesting properties of actions on Teichmuller
spaces (which are deformation spaces of representations of surface
groups). 

The problem with generalizing such theories, is that they are intrinsically
extremely geometrical, and so one would like to know if and how we
can define all such objects in a group theoretical way.

In the case of the mapping class group, we have its building blocks
which are Dehn twists around simple closed curves, which one of the
main results in this paper is explaining what their generalization
to Demuskin groups is, and showing that indeed they make sense in
such a setting.

As we explain in fact \ref{fact:scc and split}, one can think of
curves on a surface, as splittings of $\pi_{1}(S)$. By splitting
of a group $G$, we mean writing $G$ as a fundamental group of some
graph of groups (and not in the sense of split exact sequences). The
main splittings we will focus on, are 1-edge graphs of groups (which
are in a sense the building blocks of all graphs of groups). These
correspond to either an amalgamated free product or to an HNN extension.

To illustrate this, look at the following set of curves on a genus
3 surface

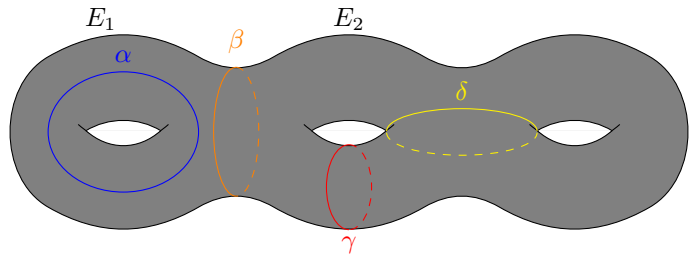
\begin{center} 
\begin{tikzpicture} 
\filldraw[fill=gray](0,1) to[out=30,in=150] (2,1) to[out=-30,in=210] (3,1) to[out=30,in=150] (5,1) to[out=-30,in=210] (6,1) to[out=30,in=150] (8,1) to[out=-30,in=30] (8,-1) to[out=210,in=-30] (6,-1) to[out=150,in=30] (5,-1) to[out=210,in=-30] (3,-1) to[out=150,in=30] (2,-1) to[out=210,in=-30] (0,-1) to[out=150,in=-150] (0,1);
\draw[smooth] (0.4,0.1) .. controls (0.8,-0.25) and (1.2,-0.25) .. (1.6,0.1);
\filldraw[fill=white][smooth] (0.5,0.02) .. controls (0.8,-0.25) and (1.2,-0.25) .. (1.5,0.02);
\filldraw[fill=white][smooth] (0.5,0.015) .. controls (0.8,0.2) and (1.2,0.2) .. (1.5,0.015);
\draw[smooth] (3.4,0.1) .. controls (3.8,-0.25) and (4.2,-0.25) .. (4.6,0.1);
\filldraw[fill=white][smooth] (3.5,0.02) .. controls (3.8,-0.25) and (4.2,-0.25) .. (4.5,0.02); 
\filldraw[fill=white][smooth] (3.5,0.015) .. controls (3.8,0.2) and (4.2,0.2) .. (4.5,0.015); 
\draw[smooth] (6.4,0.1) .. controls (6.8,-0.25) and (7.2,-0.25) .. (7.6,0.1); 
\filldraw[fill=white][smooth] (6.5,0.02) .. controls (6.8,-0.25) and (7.2,-0.25) .. (7.5,0.02);
\filldraw[fill=white][smooth] (6.5,0.015) .. controls (6.8,0.2) and (7.2,0.2) .. (7.5,0.015);
\draw [color=red](4.0,-0.17) arc(270:90:0.3 and -1.13/2);
\draw[color=red][dashed] (4.0,-0.17) arc(270:450:0.3 and -1.13/2);
\draw [color=yellow](4.5,0.01) arc(0:180: -1 and 0.3);
\draw [color=yellow][dashed](4.5,-0.01) arc(0:-180: -1 and 0.3);
\draw [color=blue](1,0) circle (1 and 0.8);
\draw [color=orange](2.5,0.86) arc(270:90:0.3 and -0.857);
\draw[color=orange][dashed] (2.5,0.86) arc(270:450:0.3 and -0.857);
\node at (1,1) {$ \color{blue} \alpha $}; 
\node at (2.5,1.2) {$ \color{orange} \beta $}; 
\node at (5.5,0.55) {$ \color{yellow} \delta $}; 
\node at (4,-1.5) {$\color{red} \gamma $};
\node at (0.7,1.5) {$ E_1 $};
\node at (4,1.5) {$ E_2 $};
\end{tikzpicture} \end{center}

the above will correspond to the following graph of groups

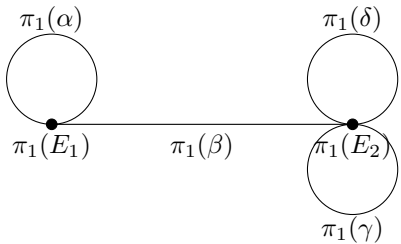
\begin{center} \begin{tikzpicture}
\filldraw  

(4,0) circle (2pt) node(1)[align=center, below] {$\pi_1(E_1)$} 
-- node[align=center, below] {$\pi_1(\beta)$} 
(8,0) circle (2pt) node[align=center,  below] {$\pi_1(E_2)$} ; 
\draw (8,0) arc(-90:270:0.6 and 0.6);
\node at (8,1.4) {$\pi_1(\delta)$};
\draw (8,0) arc(-90:270:0.6 and -0.6);
\node at (8,-1.4) {$\pi_1(\gamma)$};
\draw (4,0) arc(-90:270:0.6 and 0.6);
\node at (4,1.4) {$\pi_1(\alpha)$};
\end{tikzpicture}\end{center}

$ $

Next, given a splitting one can associate a Dehn twist to it, which
is an automorphism of the group (this need not be outer in general).

Using these we try and follows the ideas presented for surfaces, and
rewrite the proofs in a similar manner to Demuskin groups, but now
using the language of splittings of a group rather than of curves
on a surface.

The first section explains what is known and not known already, in
arithmetic topology, anabelian geometry and mapping class groups of
surfaces.

The second section gives a more in depth look into splittings of a
group, and the more combinatorial approach of graphs of groups and
Bass-Serre theory.

With the main definitions being that of splittings (definition \ref{def:Group splitting})
and of a Dehn twist (definition \ref{def:Dehn autom}), which will
be the objects of study for section \ref{sec:Main results}.
\begin{rem}
Let us emphasis, that although we restrict ourselves to surface groups
and p-adic fields in this thesis, the point of view of graphs of groups
and splittings can be used to bring geometry to many more groups,
and we think that the point of view we give here can be used to further
the global analogy of arithmetic topology as well.
\end{rem}

In the final section we focus on Demuskin groups (which as mentioned
above, are a family of pro-p groups, which include $G_{K}(p)$ and
the pro-p completion of $\pi_{1}(S)$) and show our

\textbf{Main Results}:

$ $

$ $
\begin{thm*}
Let $G$ be a Demushkin group with $d'$ generators and orientation
character which vanishes mod $p^{r}$ but not mod $p^{r+1}$, and
suppose that $p\neq2$. Then there exists a family of discrete groups
$\{\mathcal{G}_{r'}\}_{r'\geq r}$, such that every splitting of $G$
over $\mathbb{Z}_{p}$ , comes from a splitting over $\mathbb{Z}$
of one of the $\mathcal{G}_{r'}$. Furthermore, the splitting of $G$
over $\mathbb{Z}_{p}$ have a ``nice'' form, in the sense that they
are the pro-p completions of the splittings described in in definition
\ref{def:D split and aut}. (For a more precise statement, see theorem
\ref{thm:curve complex bijection})
\end{thm*}
$ $
\begin{thm*}
Let $G$ and $\mathcal{G}_{r'}$ be as in the previous theorem, then
we have an injection $Out_{D}(\mathcal{G}_{r'})\hookrightarrow Out(G)$
(where the subscript $D$, denotes the subgroup generated by Dehn
twists coming from ``nice'' splittings, as defined in definition
\ref{def:D split and aut})
\end{thm*}
$ $
\begin{thm*}
Let $G$ be a Demuskin group, let $\alpha$ be a splitting of $G$
over $\mathbb{Z}_{p}$ and let $T_{\alpha}$ denote the Dehn twist
of $\alpha$, which is an automorphism of $G$ (as defined in \ref{def:general Dehn autom}).
Then for every integer $k$, we have that the image of $T_{\alpha}^{k}$
in $Out(G)$ is non-trivial.
\end{thm*}
$ $
\begin{cor*}
The family $\{\mathcal{G}_{r'}\}_{_{r'\geq r}}$ has infinitely many
isomorphism classes, on the other hand, the pro-l completion of all
of these groups are free pro-l groups on $d'-1$ generators, and the
pro-p completion is the same Demuskin group (namely the one on $d'$
generators, and orientation character which vanishes mod $p^{r}$
but not mod $p^{r+1}$).

$ $

$ $
\end{cor*}
Note that the above theorem gives a large family of non trivial example
of outer automorphisms of $G_{K}(p)$ (which will include the analogue
of the Outer automorphism appearing in \cite{neukirch2013cohomology},
defined in definition \ref{def: neuk outer}, as explained in example
\ref{exa:neuk outer is dehn}).

We also note that all the proofs will also have a geometric flavour
to them, and can be seen as direct generalizations of classical results
about Dehn twists of surfaces.

The proofs of corollary \ref{cor:p-eff splits} and theorems \ref{thm: equivariant CC map},
\ref{thm:curve complex bijection} closely follow the ideas of Wilkes
(\cite{wilkes2020classification,wilkes2017virtual}) and generalize
them to , the main differences being the need to keep track of more
discrete group completing to the given pro-$p$ group, and also not
having some of the geometric tools available for surface groups.

$ $
\begin{rem*}
Other than the theorems mentioned above, the author thinks that another
important contribution of this thesis is by expanding the arithmetic
topology dictionary, and by giving a new and more uniform way for
discussing surfaces and p-adic fields. In a sense, the author sees
this thesis as more of a 'proof of concept' of this way of thinking,
and as a first part in a bigger project of extending results about
surfaces to p-adic fields (and vice versa), which will include further
papers, some of which are already in writing.

Thus, we finish the introduction with a table to help readers find
the corresponding guiding case for surfaces in ``A primer on mapping
class groups'' by Farb and Margalit \cite{farb2011primer}:
\end{rem*}
\begin{center}\begin{small}

\begin{tabular}{|c|c|c|}
\hline 
 & $\begin{array}{c}
\text{Primer on Mapping}\\
\text{class groups }
\end{array}$ \cite{farb2011primer} & Current paper\tabularnewline
\hline 
\hline 
Mapping class group & section 2.1 & $Out(G)$\tabularnewline
\hline 
simple closed curve & section 1.2.2 & splitting of $\pi_{1}$ (def. \ref{def:Group splitting})\tabularnewline
\hline 
intersection of curves & section 1.2.3 & Definition \ref{def:intersection}\tabularnewline
\hline 
The change of coordinates principle & section 1.3 & Theorem \ref{thm:curve complex bijection}\tabularnewline
\hline 
Definition of Dehn twists & section 3.1.1 & Definition \ref{def:Dehn autom}\tabularnewline
\hline 
Formula for the conjugate of a Dehn twist & Fact 3.7 & Lemma \ref{lem:cong formula}\tabularnewline
\hline 
Dehn twists are elements of $Out(\pi_{1})$ & proposition 3.1 & Corollary \ref{cor:main dehn outer}\tabularnewline
\hline 
Dehn twists of non-intersection commute & Fact 3.9 & Lemma \ref{lem:no inters commute}\tabularnewline
\hline 
Curve complex & section 4.1 & Definition \ref{def:Curve complex}\tabularnewline
\hline 
\end{tabular}

\end{small}\end{center}

\subsubsection*{Remark on the structure of the paper}

The current paper is aimed towards both people in group theory and
Number theory, due to that we recall certain basics in both fields.

Subsections \ref{subsec:Arithmetic-topology} and \ref{subsec:anabelian}
give the background on the Number theory motivations that led to the
current paper.

Subsection \ref{subsec:Mapping-class-groups} recalls the basics of
mapping class groups of surfaces, and a general picture of the type
of results the author hopes to get for p-adic fields

Section \ref{sec:Bass-serre} recalls and explains the needed Bass-Serre
theory of graphs of groups, and subsection \ref{subsec:Profinite Bass-serre}
covers the profinite setting for graphs of groups.

Readers interested in just the main theorems presented in this paper,
can skip to section \ref{sec:Main results}, which begins with the
definitions of Poincare groups, followed by the proofs of all the
main theorems.

\subsection*{Acknowledgments}

The current paper was written as part of the author's DPhil at the
University of Oxford, under the supervision of Prof. Konstantin Ardakov
and Prof. Minhyong Kim, and with financial support from the Mathematical
Institute of the University of Oxford.

The author would like to thank his supervisors and the University
of Oxford, and the many people with whom he had many helpful conversations
throughout his DPhil and before it.

In particular he would like to thank his viva examiners, Kobi Kremnitzer
and Tomer Schlank, for helpful comments and corrections. He would
also like to thank Jonathan Fruchter, Ido Grayevsky and Alex Lubotzky,
for helpful discussions on materials related to the current paper,
and to Dani Wise for suggesting the use of Whitehead's algorithm in
the proof of the final corollary.

The author would also like to thank Uri Bader and Tsachik Gelander,
and their Midrasha on groups, where he learned most of his geometric
group theory background.

\tableofcontents{}

\newpage

\section{General Motivation}

\subsection{\label{subsec:Arithmetic-topology}Arithmetic topology }

As stated in the introduction, arithmetic topology draws an analogy
between primes and knots, surfaces and p-adic fields and number fields
and 3-manifolds.

In this subsection we will give a very brief overview, for much more
information see \cite{morishita2011knots}.

For example, class field theory for a number field $K$ can be stated
as a sort of 3-dimensional Poincare duality in the $\acute{e}$tale
cohomology of $Spec(O_{K})$ (\cite{mazur1973notes}). 

Mazur pointed out this line of thought and the analogy between a knot
and prime in \cite{mazur1963remarks}. 

Looking at $\acute{e}$tale homotopy groups for prime fields $\mathbb{F}_{p}=\mathbb{Z}/p\mathbb{Z}$
we get: groups: $\pi_{1}^{et}(Spec(\mathbb{F}_{p}))=\hat{\mathbb{Z}},\pi_{i}^{et}(Spec(\mathbb{F}_{p})=0\text{ (i\ensuremath{\geq2)}}$
(where $\hat{\mathbb{Z}}$ is the pro-finite completion of $\mathbb{Z}$)
and so we regard $Spec(\mathbb{F}_{p})$ as an arithmetic analogue
of a circle $S^{1}$. 

Since $Spec(\mathbb{Z})$ has $\acute{e}$tale cohomological dimension
3 (up to 2-torsion) and $\pi_{1}^{et}(Spec(\mathbb{Z}))=1$, it can
be viewed as the 3-sphere.

We thus get that the embeddings $Spec(\mathbb{F}_{p})\hookrightarrow Spec(\mathbb{Z})$
can be viewed as the arithmetic analogue of a knot, i.e. an embedding
$S^{1}\hookrightarrow S^{3}$. 

The analogies between knots and primes, 3-manifolds and number rings
were later studied more by Kapranov (\cite{kapranov1996analogies}), Reznikov (\cite{reznikov1997three},\cite{reznikov2000embedded}), Morishita (\cite{morishita2011knots}) and many others.

In the view of the analogy above, a knot group $G_{K}=\pi_{1}(S^{3}\backslash K)$)
corresponds to a \textquotedblleft prime group\textquotedblright{}
$G_{\{(p)\}}=\pi_{1}^{et}(Spec(\mathbb{Z})\backslash\{(p)\})$.

More generally, a link $L$ corresponds to a finite set of primes,
$S$, and the link group $G_{L}=\pi_{1}(S^{3}\backslash L$) corresponds
to $G_{S}=\pi_{1}^{et}(Spec(\mathbb{Z})\backslash S)$, the Galois
group of the maximal Galois extension of $\mathbb{Q}$ unramified
outside $S\cup\{\infty\}$. Under this we see that the absolute Galois
group of $\mathbb{Q}$ can be seen as the fundamental group of the
3-sphere minus infinitely many knots. 

One has Kapranov-Reznikov-Mazur-Morishita dictionary:

\begin{center}\begin{small}

\begin{tabular}{|c|c|}
\hline 
Arithmetic & Topology\tabularnewline
\hline 
\hline 
$\begin{array}{c}
Spec(O_{K})\cup\{\infty\}\\
Spec(\mathbb{Z})\cup\{\infty\}
\end{array}$  & $\begin{array}{c}
\text{3 manifold }M\\
S^{3}
\end{array}$\tabularnewline
\hline 
$\begin{array}{c}
Spec(\mathbb{F}_{\mathfrak{p}})\subset Spec(O_{K})\\
\text{primes \ensuremath{\mathfrak{p}_{1},...,\mathfrak{p}_{n}}}\\
\\
\end{array}$ & $\begin{array}{c}
\text{Knot \ensuremath{K\subset M}}\\
\text{Link }K_{1}\cup...\cup K_{n}\\
\text{}
\end{array}$\tabularnewline
\hline 
$\begin{array}{c}
\mathfrak{p}\text{-adic integers }Spec(O_{K_{\mathfrak{p}}})\\
\pi_{1}^{et}(Spec(O_{K_{\mathfrak{p}}}))=<\sigma>
\end{array}$ & $\begin{array}{c}
\text{tubular neighborhood of K, }V(K)\\
\pi_{1}(V(K))=<\beta>\text{ }
\end{array}$\tabularnewline
\hline 
$\begin{array}{c}
\mathfrak{p}\text{-adic field }Spec(K_{\mathfrak{p}})\\
\pi_{1}^{et,tame}(Spec(K_{\mathfrak{p}}))=<\sigma,\tau:\tau^{p-1}[\tau,\sigma]=1>
\end{array}$ & $\begin{array}{c}
\text{the torus }\partial V(K)\\
\pi_{1}(\partial V(K))=<\alpha,\beta:[\alpha,\beta]=1>
\end{array}$\tabularnewline
\hline 
class group $H_{K}$ & $H_{1}(M,\mathbb{Z})$\tabularnewline
\hline 
units $O_{K}^{\times}$ & $H_{2}(M,\mathbb{Z})$\tabularnewline
\hline 
$\begin{array}{c}
\text{Maximal Galois group unramified outside S:}\\
\pi_{1}^{et}(Spec(O_{K})\backslash\{(\mathfrak{p}_{1},...,\mathfrak{p}_{n})\}
\end{array}$ & $\begin{array}{c}
\text{Link group:}\\
\pi_{1}(M\backslash K_{1}\cup..\cup K_{n})
\end{array}$\tabularnewline
\hline 
Legendre symbol & Linking number\tabularnewline
\hline 
p-adic L-function & Alexander polynomial of $K$\tabularnewline
\hline 
\end{tabular}

\end{small}\end{center}

Under this, a lot of the local structure is not accounted for. For
example it only looks at the tamely ramified part, and always corresponds
a p-adic field to a genus 1 surface. 

We hope to enrich the above dictionary in this paper to take into
account for all the Galois group of a p-adic field, and to associate
different genus surfaces, depending on the degree of the extension
over $\mathbb{Q}_{p}$.

\subsection{\label{subsec:anabelian}Rigidity and anabelian geometry}

Another interesting setting where the analogy between arithmetic and
topology continues, and is of special interest to us is Mostow rigidity
theorem and Uchida-Neukirch \cite{neukirch1969kennzeichnung}\cite{uchida1976isomorphisms}
theorem in anabelian geometry.

First we have the following rigidity theorem of Mostow:
\begin{thm}
Let ${\displaystyle M}$ and $N$ be complete finite-volume hyperbolic
manifolds of dimension $n\geq3$.

If $\pi_{1}(M)\cong\pi_{1}(N)$, then $M$ and $N$ are isometric.
\end{thm}

Similarly in anabelian geometry, one studies how much information
about a space $X$ is contained already in its first $\acute{e}$tale
homotopy group $\pi_{\acute{e}t}^{1}(X,x)$.

The term \textquotedblleft anabelian\textquotedblright{} is comes
from the idea that \textquotedblleft the less abelian $\pi_{\acute{e}t}^{1}(X,x)$
is, the more information it carries about $X$\textquotedblright .

We want $\pi_{\acute{e}t}^{1}(X,x)$ to be very not abelian, one example
that helps see why we want this is that of elliptic curves.

Let $E$ be an elliptic curve over an algebraically closed field of
characteristic 0, it has $\pi_{\acute{e}t}^{1}(E,x)=\hat{\mathbb{Z}}^{2}$,
where we take any base geometric point $x$. Thus we have, $\pi_{\acute{e}t}^{1}(E,x)$
does not distinguish between elliptic curves. Similarly, for finite
fields $\mathbb{F}_{p}$ we have $\pi_{\acute{e}t}^{1}(\mathbb{F}_{p},x)=\hat{\mathbb{Z}}$,
and thus again, $\pi_{\acute{e}t}^{1}$ does not distinguish between
finite fields.

An algebraic variety whose isomorphism class is entirely determined
by $\pi_{\acute{e}t}^{1}(X,x)$ is called an anabelian variety.

Uchida \cite{uchida1976isomorphisms} and Neukirch \cite{neukirch1969kennzeichnung},
proved that Number fields are anabelian. This can be seen as an arithmetic
analogue of Mostow\textquoteright s rigidity theorem:
\begin{thm}
Let ${\displaystyle F}$ and $F'$ be number fields.

If $\pi_{\acute{e}t}^{1}(F)=Gal(\bar{F}/F)\cong Gal(\bar{F'}/F')=\pi_{\acute{e}t}^{1}(F')$,
then $F'$ and $F$ are isomorphic, and we have a bijection:

\[
Isom_{\mathbb{Q}}(F,F')\rightarrow Out(G{}_{F},G_{F'})
\]
\end{thm}

A similar statement was proven for p-adic fields by Mochizuki \cite{mochizuki1997version}:
\begin{thm}
\label{thm: Moch anab}Let $K$ and $K'$ be finite extensions of
$\mathbb{Q}_{p}$. Let $Isom_{\mathbb{Q}_{p}}(K,K')$ denote the set
of $\mathbb{Q}_{p}$-algebra isomorphisms of $K$ with $K'$.

Let $Out_{Filt}(G_{K},G_{K'})$ denote the set of outer isomorphisms
of filtered groups between the absolute Galois groups of $K$ and
$K'$ equipped with the filtrations defined by the higher (i.e., with
index > 0) ramification groups in the upper numbering.

Then the natural morphism
\[
Isom_{\mathbb{Q}_{p}}(K,K')\rightarrow Out_{Filt}(G{}_{K},G_{K'})
\]
 induced by \textquotedblleft looking at the morphism induced on absolute
Galois groups\textquotedblright{} is a bijection.
\end{thm}

The extra condition on the ramification filtration is a necessary
one, and there are non isomorphic p-adic fields with isomorphic absolute
Galois groups.
\begin{example}
The fields $\mathbb{Q}_{p}(\mu_{p},\sqrt[p]{p})$ and $\mathbb{Q}_{p}(\mu_{p},\sqrt[p]{p+1})$
are not isomorphic, but have isomorphic absolute Galois groups (this
example and more can be found in \cite{jarden1979characterization}). 
\end{example}

In fact there are non trivial elements in $Out(G_{K})$ for any p-adic
fields $K$, as we will explain in the next subsection.

This again fits with the topological case in the sense that p-adic
fields can be thought as analogues of surfaces.

In the surface setting we do not have Mostow's rigidity, and $Out(\pi_{1}(S))$
for a surface $S$ is a very interesting object, also known as the
mapping class group of $S$.
\begin{rem}
In 1984 Grothendieck\cite{grothendieck1997esquisse}, made the following
important conjecture which motivated the theory of anabelian geometry:
\end{rem}

\begin{conjecture}
\label{conj:Anabelian conj}All hyperbolic curves over number fields
are anabelian varieties.
\end{conjecture}

For algebraic curves over finite fields, over number fields and over
p-adic field the statement was eventually completed by Mochizuki \cite{mochizuki1996profinite}.

\subsubsection{explicit outer automorphism}

In this subsection we will go over a structure theorem for Galois
groups of p-adic fields, which will later be used to construct outer
morphisms of $G_{\mathbb{Q}_{p}}$ (this can be found in \cite{neukirch2013cohomology},
end of Chapter 7).
\begin{thm}[(Jannsen-Wingberg)]
\label{thm:Jannsen-Wingberg} Let $K$ be a p-adic field, with $N=[K:\mathbb{Q}_{p}]$,
$q$ the cardinality of the residue field of $K$, $p^{s}$ the order
of the group of all $p$-power roots of unity in $K^{tr}$, the maximal
tamely ramified extension of $K$.

The group $Gal(K^{tr}/K)$ is generated by $\sigma,\tau$ satisfying
$\sigma\tau\sigma^{-1}=\tau^{q}$, and denote by $g,h\in\mathbb{Z}_{p}$
the image of $\sigma,\tau$ respectively, under the cyclotomic character.

$ $

Then the group $G_{K}$ has $N+3$ generators (where $N=[K:\mathbb{Q}_{p}]$),
$\sigma,\tau,x_{0},...,x_{N}$, with the following relations/conditions:

A) The closed normal subgroup generated by $x_{0},...,x_{N}$ is a
pro-$p$-group.

B) The elements $\sigma,\tau$ satisfy the tame relation $\sigma\tau\sigma^{-1}=\tau^{q}$.

C) i) If N is even $x_{0}^{\sigma}=<x_{0},\tau>^{g}x_{1}^{p^{s}}[x_{1},x_{2}]...[x_{N-1},x_{N}]$

ii) If N is odd $x_{0}^{\sigma}=<x_{0},\tau>^{g}x_{1}^{p^{s}}[x_{1},y_{1}][x_{2},x_{3}]...[x_{N-1},x_{N}]$

where $<x_{0},\tau>=(x_{0}^{h^{p-1}}\tau x_{0}^{h^{p-2}}...x_{0}^{h}\tau)^{\frac{\pi}{p-1}}$
and where $y_{1}$ is a certain element in the subgroup generated
by $x_{1},\sigma,\tau$.
\end{thm}

\begin{rem}
In a sense this would mean that this should correspond to a surface
of genus $\frac{N}{2}+1$.
\end{rem}

For the full proof we refer to \cite{jannsen1982struktur} and \cite{diekert1984absolute},
it is based on a theory of H. Koch \cite{koch1978galois} which axiomizes
the fact that for every finite, tamely ramified extension $L/K$ the
group $G_{L}(p)$ is a Demuskin group. 

Group theoretically, it uses the following 3 axioms/conditions:

Let $H$ is an open normal subgroup of $\mathcal{G}=Gal(K^{tr}/K)$
contained in the kernel of $\chi_{tr}$ (the cyclotomic character
from $\mathcal{G}$), so that $\mu_{p^{s}}$ is contained in the fixed
field $L$ of $H$, and denote by $G=\mathcal{G}/H=Gal(L/K)$.

We have that $G_{L}(p)=Gal(L(p)/L)$, is the Galois group of the maximal
p-extension of $L$, it is the maximal pro-p-factor group of $G_{L}$
and is a Demuskin group. 

For these groups we have the following known properties. 

I. - $dimH^{1}(G_{L},\mathbb{F}_{p})<\infty$, $dimH^{2}(G_{L},\mathbb{F}_{p})=1$
and $H^{1}(G_{L},\mathbb{F}_{p})\times H^{1}(G_{L},\mathbb{F}_{p})\overset{\cup}{\rightarrow}H^{2}(G_{L},\mathbb{F}_{p})$
is a non degenerate anti-symmetric bilinear form. 

II. - Looking at $H^{1}(H,\mathbb{F}_{p})$ as a 1-dimensional subspace
of the symplectic space $H^{1}(G_{L},\mathbb{F}_{p})$ we have an
isomorphism of $G$-modules $H^{1}(H,\mathbb{F}_{p})^{\perp}/H^{1}(H,\mathbb{F}_{p})\cong\mathbb{F}_{p}[G]^{n}$. 

With respect to the induced non-degenerate bilinear form this $G$-module
is a direct sum of two totally isotropic $G$-submodules. 

III. - $(G_{L}(p)^{ab})_{tor}\cong\mu_{p^{s}}$ as a $G$-module.
\begin{rem}
\label{rem:Condition explain}Condition $I$ expresses the fact that
$G_{L}(p)$ is a Demuskin group.

By class field theory $H^{1}(G_{L},\mathbb{F}_{p})$ is dual to $L^{*}/L^{*^{p}}=(\pi)/(\pi)^{p}\times U^{1}/(U^{1})^{p}$
where $\pi$ is a prime element of $K$ and $U^{1}$ is the group
of principal units of $L$. 

The cup product goes to the Hilbert symbol on $L^{*}/L^{*^{p}}$ and
$U^{1}/(U^{1})^{p}$ contains$H^{1}(H,\mathbb{F}_{p})^{\perp}/H^{1}(H,\mathbb{F}_{p})$
as a subspace of dimension 1 which is isomorphic to $\mathbb{F}_{p}[G]$.

The last isomorphism is a result of Iwasawa. The assertion on the
the direct sum of totally isotropic submodules property is due to
Koch.
\end{rem}

In \cite{neukirch2013cohomology} (chapter 7 section 5) an example
of a non trivial outer automorphisms of $G_{K}$ is constructed, thus
showing that the extra filtration preserving condition in theorem
\ref{thm: Moch anab} is necessary.

It is easier to describe and show it is outer for $N=[K:\mathbb{Q}_{p}]>1$.

Let $\sigma,\tau,x_{0},...,x_{N}$ be the generators of $G_{K}$ described
in theorem \ref{thm:Jannsen-Wingberg}.
\begin{defn}
\label{def: neuk outer}Define $\psi:G_{K}\rightarrow G_{K}$ by $\psi(y)=y$
for $y=\sigma,\tau,x_{0},...,x_{N-1}$ and $\psi(x_{N})=x_{N}x_{N-1}$.
\end{defn}

\begin{rem}
In example \ref{exa:neuk outer is dehn} we see that for the right
choice of splitting, the above construction is also a Dehn twist in
the sense of definition \ref{def:Dehn autom}.
\end{rem}

This is then an automorphism of $G_{K}$,since $[x_{N-1};x_{N}]=[x_{N-1};x_{N}x_{N-1}]$
and so the generators $\sigma,\tau,x_{0},...,x_{N-1},x_{N}x_{N-1}$
satisfy both relations if $N>1$. 

Now suppose that is an inner automorphism, i.e. there is an element
$\rho\in G_{K}$ such that $\psi(z)=z^{\rho}$ for all $z\in G_{K}$. 

Denote by $V_{K}^{i}$ the $i$-th term of the $p$-central series
of the ramification group $V_{K}$ with $V_{K}^{1}=V_{K}$.

Since $x_{N-1}^{\rho}=\psi(x_{N-1})=x_{N-1}$ and since $x_{n-1}V_{K}^{2}$
generates a free $\mathbb{F}_{p}[[G_{K}]]$ module in $V_{K}/V_{K}^{2}$
, by \cite{jannsen1982struktur}, we obtain that $\rho\in V_{K}$.

It follows that $x_{N}V_{K}^{2}=x_{N}x_{N-1}V_{K}^{2}$ which is a
contradiction. Thus, if $N=[K:\mathbb{Q}_{p}]>1$, we have constructed
a nontrivial outer automorphism of $G_{K}$. 

The case $K=\mathbb{Q}_{p}$ is more difficult. It is constructed
in a similar manner, but also uses the following result from \cite{wingberg1982eindeutigkeitssatz}.

\begin{thm}
Let $\psi_{3}:G_{K}/V_{K}^{3}\rightarrow G_{K}/V_{K}^{3}$ be an automorphism
of $G_{K}/V_{K}^{3}$ which induces the identity on the factor group
$G_{K}/V_{K}$.

Then there exists an automorphism of $G_{K}$ which coincides with
$\psi_{3}$ modulo $V_{K}^{2}$. 
\end{thm}

\subsection{\label{subsec:Mapping-class-groups}Mapping class groups of surface}

We would like to have some arithmetic theory similar to the theory
of mapping class groups, and more specifically the mapping class group
of a surface (for more information on the subject we refer to \cite{fathi2012thurston}.

Let $S$ be a connected, closed, orientable surface and ${\displaystyle \mathrm{Homeo}^{+}(S)}$
the group of orientation-preserving, or positive, homeomorphisms of
$S$ ((that restrict to the identity on $\partial S$ if $S$ has
a boundary).

This group has a natural topology, the compact-open topology. It can
be defined easily by a distance function: if we are given a metric
$d$ on $S$ inducing its topology then the function defined by
\[
{\displaystyle \delta(f,g)=\sup_{x\in S}\left(d(f(x),g(x)\right)}
\]
is a distance inducing the compact-open topology on ${\displaystyle \mathrm{Homeo}^{+}(S)}$. 

The connected component of the identity for this topology is denoted
${\displaystyle \mathrm{Homeo}_{0}(S)}$. By definition it is equal
to the homeomorphisms of $S$ which are isotopic to the identity and
it is a normal subgroup of the group of positive homeomorphisms.
\begin{defn}
\label{def:MCG}We define the mapping class group of $S$ as the group
\[
{\displaystyle \mathrm{MCG}(S)=\mathrm{Homeo}^{+}(S)/\mathrm{Homeo}_{0}(S)}
\]

Modifying the definition to include all homeomorphisms we obtain the
extended mapping class group ${\displaystyle \mathrm{MCG}^{\pm}(S)}$.
\end{defn}

Note that these groups are countable, and the extended mapping class
group ${\displaystyle \mathrm{MCG}^{\pm}(S)}$, contains the mapping
class group as a subgroup of index 2.
\begin{rem}
This definition can also be made in the differentiable category: if
we replace all instances of \textquotedbl homeomorphism\textquotedbl{}
above with \textquotedbl diffeomorphism\textquotedbl{} we obtain
the same group, that is the inclusion ${\displaystyle \mathrm{Diff}^{+}(S)\subset\mathrm{Homeo}^{+}(S)}$
induces an isomorphism between the quotients by their respective identity
components.
\end{rem}

If $S$ is closed and $f$ is a homeomorphism of $S$ then we can
define an automorphism ${\displaystyle f_{*}}$ of the fundamental
group ${\displaystyle \pi_{1}(S,x_{0})}$ as follows: fix a path $\gamma$
between $x_{0}$ and $f(x_{0})$ and for a loop $\alpha$ based at
${\displaystyle x_{0}}$ representing an element ${\displaystyle [\alpha]\in\pi_{1}(S,x_{0})}$
define ${\displaystyle f_{*}([\alpha])}$ to be the element of the
fundamental group associated to the loop ${\displaystyle {\bar{\gamma}}*f(\alpha)*\gamma}$.

This automorphism depends on the choice of $\gamma$ , but only up
to conjugation. Thus we get a well-defined map $\Phi$ from ${\displaystyle \mathrm{Homeo}(S)}$
to the outer automorphism group ${\displaystyle \mathrm{Out}(\pi_{1}(S,x_{0}))}$.
\begin{fact}
\label{fact:Dehn-Nilsen}The map $\Phi$ is a morphism and its kernel
is exactly the subgroup ${\displaystyle \mathrm{Homeo}_{0}(S)}$. 

The Dehn--Nielsen--Baer theorem \cite{dehn1987papers,nielsen1927untersuchungen}
states that it is in addition surjective.

In particular, it implies that the extended mapping class group ${\displaystyle \mathrm{MCG}^{\pm}(S)}$
is isomorphic to the outer automorphism group ${\displaystyle \mathrm{Out}(\pi_{1}(S))}$.

The image of the mapping class group is an index 2 subgroup of the
outer automorphism group, which can be characterized by its action
on homology.
\end{fact}

A closed curve in a surface S is a continuous map $S^{1}\rightarrow S$,
it is simple if the map in injective. A closed curve is called essential
if it is not homotopic to a point, a puncture, or a boundary component.
We will usually identify a closed curve with its image in $S$.

Consider the annulus $A=S^{1}\times[0,1],$ we embed it in the plane
with polar coordinates via $(\theta,t)\mapsto(\theta,t+1)$, and take
the orientation induced by the standard orientation of the plane.
Define $T:A\rightarrow A$ be the \textquotedblleft left twist map\textquotedblright{}
of $A$ given by the formula $T(\theta,t)=(\theta+2\pi t,t)$ (note
that taking instead $\theta-2\pi t$ would give a right twist).

The map $T$ is an orientation-preserving homeomorphism that fixes
the boundary of $A$ point-wise. 

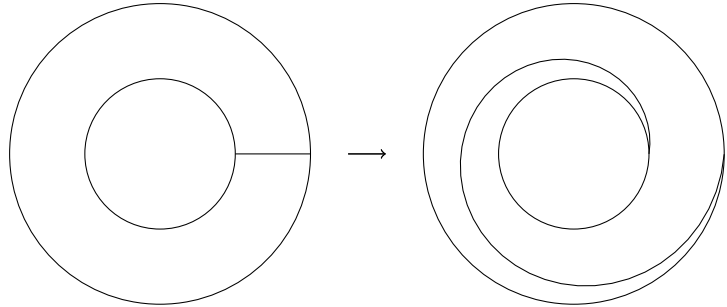
\begin{center} \begin{tikzpicture}
\draw(0,0) circle (1);
\draw(0,0) circle (2); 
\draw (1,0) -- (2,0);
\draw [->](2.5,0) --(3,0);
\draw(5.5,0) circle (1);
\draw(5.5,0) circle (2); 
\draw [domain=1:2, samples=50] plot ({\x*cos(2*pi*\x r)+5.5}, {\x*sin(2*pi*\x r)});
\draw [->](2.5,0) --(3,0);
\end{tikzpicture}\end{center}
\begin{defn}
\label{def:Dehn twist}Let $\alpha$ be simple closed curve on $S$
and let $N$ be a regular neighborhood of $\alpha$ and take $\phi$
be an orientation preserving homeomorphism then there is a homeomorphism
$f$ from $A$ to the canonical annulus ${\displaystyle A_{0}}$,
sending $C$ to a circle with the counterclockwise orientation.

Define a homeomorphism $T_{\alpha}$ of $S$ as follows: on ${\displaystyle S\setminus N}$
it is the identity, and on $N$ it is equal to ${\displaystyle \phi^{-1}\circ T\circ\phi}$.

The class of ${\displaystyle T_{\alpha}}$ in the mapping class group
${\displaystyle \mathrm{MCG}(S)}$ does not depend on the choice of
$\phi$ made above, and the resulting element is called the Dehn twist
about $\alpha$.

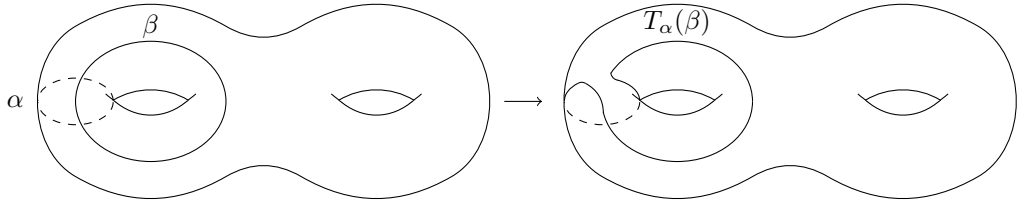
\begin{center} 
\begin{tikzpicture}
\draw (-3,1) to[out=30,in=150] (-1,1) to[out=-30,in=210] (0,1) to[out=30,in=150] (2,1) to[out=-30,in=30] (2,-1) to[out=210,in=-30] (0,-1) to[out=150,in=30] (-1,-1) to[out=210,in=-30] (-3,-1) to[out=150,in=-150] (-3,1);
\draw[smooth] (-2.6,0.1) .. controls (-2.2,-0.25) and (1.2-3,-0.25) .. (1.6-3,0.1);
\filldraw[fill=white][smooth] (0.5-3,0.02) .. controls (0.8-3,-0.25) and (1.2-3,-0.25) .. (1.5-3,0.02);
\filldraw[fill=white][smooth] (0.5-3,0.015) .. controls (0.8-3,0.2) and (1.2-3,0.2) .. (1.5-3,0.015);
\draw[smooth] (3.4-3,0.1) .. controls (3.8-3,-0.25) and (4.2-3,-0.25) .. (4.6-3,0.1);
\filldraw[fill=white][smooth] (3.5-3,0.02) .. controls (3.8-3,-0.25) and (4.2-3,-0.25) .. (4.5-3,0.02); 
\filldraw[fill=white][smooth] (3.5-3,0.015) .. controls (3.8-3,0.2) and (4.2-3,0.2) .. (4.5-3,0.015); 
\draw (1-3,0) circle (1 and 0.8);
\draw [dashed](0.5-3,0.01) arc(180:0: -0.5 and 0.3);
\draw [dashed](0.5-3,-0.01) arc(180:360: -0.5 and 0.3);
\node at (1-3,1) {$  \beta $}; 
\node at (-0.8-3,0) {$  \alpha $}; 
\draw [->] (5.7-3,0) -- (6.2-3,0);

\draw (-3+7,1) to[out=30,in=150] (-1+7,1) to[out=-30,in=210] (0+7,1) to[out=30,in=150] (2+7,1) to[out=-30,in=30] (2+7,-1) to[out=210,in=-30] (0+7,-1) to[out=150,in=30] (-1+7,-1) to[out=210,in=-30] (-3+7,-1) to[out=150,in=-150] (-3+7,1);
\draw[smooth] (-2.6+7,0.1) .. controls (-2.2+7,-0.25) and (1.2+4,-0.25) .. (1.6+4,0.1);
\filldraw[fill=white][smooth] (0.5+4,0.02) .. controls (0.8+4,-0.25) and (1.2+4,-0.25) .. (1.5+4,0.02);
\filldraw[fill=white][smooth] (0.5+4,0.015) .. controls (0.8+4,0.2) and (1.2+4,0.2) .. (1.5+4,0.015);
\draw[smooth] (3.4+4,0.1) .. controls (3.8+4,-0.25) and (4.2+4,-0.25) .. (4.6+4,0.1);
\filldraw[fill=white][smooth] (3.5+4,0.02) .. controls (3.8+4,-0.25) and (4.2+4,-0.25) .. (4.5+4,0.02); 
\filldraw[fill=white][smooth] (3.5+4,0.015) .. controls (3.8+4,0.2) and (4.2+4,0.2) .. (4.5+4,0.015); 
\draw (6,0) arc(0:152: 1 and 0.8);
\draw (6,0) arc(0:-170: 1 and 0.8);
\draw (0.5+4,0.01) arc(180:114: -0.5 and 0.3);
\draw (0.5+3,0.01) arc(0:50: -0.5 and 0.3);
\draw (4.21,0.282) to[bend left] (4.12,0.36);
\draw (4.12,0.36) to[bend left] (4.12,0.38);
\draw (3.68,0.24) to[bend left] (3.78,0.24);
\draw (3.78,0.24) to[bend left](4.017,-0.15);
\draw [dashed](0.5+4,-0.01) arc(180:360: -0.5 and 0.3);
\node at (1+4,1.05) {$  T_\alpha (\beta) $}; 
\end{tikzpicture} 
\end{center}
\end{defn}

\begin{rem}
One would like to think of the map $\psi$ from Definition \ref{def: neuk outer},
as some sort of arithmetic Dehn twist.
\end{rem}

If $\alpha$ is not null-homotopic this mapping class is nontrivial,
and more generally the Dehn twists defined by two non-homotopic curves
are distinct elements in the mapping class group.

In the mapping class group of the torus identified with ${\displaystyle \mathrm{SL}_{2}(\mathbb{Z})}$
the Dehn twists correspond to unipotent matrices.

For example, the matrix ${\displaystyle \begin{pmatrix}1 & 1\\
0 & 1
\end{pmatrix}}{\displaystyle }$ corresponds to the Dehn twist about a horizontal curve in the torus.

The mapping class group is generated by the subset of Dehn twists
about all simple closed curves on the surface.

The Dehn--Lickorish theorem states that it is sufficient to select
a finite number of those to generate the mapping class group. This
generalises the fact that ${\displaystyle \mathrm{SL}_{2}(\mathbb{Z})}$
is generated by the matrices ${\displaystyle \begin{pmatrix}1 & 1\\
0 & 1
\end{pmatrix},\begin{pmatrix}1 & 0\\
1 & 1
\end{pmatrix}}{\displaystyle }$. 

In particular, the mapping class group of a surface is a finitely
generated group.

The least possible cardinality of Dehn twists generating the mapping
class group of a closed surface of genus ${\displaystyle g\geq2}$
is ${\displaystyle 2g+1}$; this was proven later by Humphries.
\begin{thm}
\label{thm:Thurston-classification}(Nielsen, Thurston) An element
${\displaystyle g\in\mathrm{MCG}(S)}$ is either:

- Of finite order 

- Reducible: there exists a set of disjoint closed curves on $S$
which is preserved by the action of $g$

- Pseudo-Anosov. 
\end{thm}

The main content of the theorem is that a mapping class which is neither
of finite order nor reducible must be pseudo-Anosov (which is a dynamical
property).

\subsubsection{Generators of the mapping class group and the curve complex}

An important first step to understanding the mapping class group,
is the fact that it is generated by a finite amount of Dehn twists.

To prove that, one introduces the curve complex.
\begin{defn}
The curve complex is a simplicial complex associated to a surface
$S$.

Vertices are free homotopy classes of essential simple closed curves
on $S$, and $v_{1},...,v_{n}$ distinct vertices span a simplex if
and only if they can be homotoped to be pairwise disjoint.
\end{defn}

Masur and Minsky proved that he curve complex is hyperbolic, and they
proceeded to use that fact to derive new information about mapping
class groups. 

The curve complex is acted on nicely by the mapping class group, and
one can show that $MCG(S)$ is generated by $stab_{v}$ and $h$ where
$v,w$ are two vertices with an edge between them, and $h(v)=w$.

\subsection{Teichmuller spaces and Galois representations }

An important space associated to a surface is its Teichmuller space:
\begin{defn}
\label{def:Teich. space}Given a surface $S$ the Teichm�ller space
${\displaystyle T(S)}$ is the space of marked complex (equivalently,
conformal or complete hyperbolic) structures on $S$.
\end{defn}

These can be represented by pairs ${\displaystyle (X,f)}$ where $X$
is a Riemann surface and ${\displaystyle f:S\to X}$ a homeomorphism,
modulo a suitable equivalence relation. 

There is an action of the group ${\displaystyle \mathrm{Homeo}^{+}(S)}$
on such pairs, which descends to an action of ${\displaystyle \mathrm{MCG}(S)}$
on Teichm�ller space.

This action has many interesting properties; for example it is properly
discontinuous (though not free).

The action extends to the Thurston boundary of Teichm�ller space,
and the Nielsen-Thurston classification of mapping classes can be
seen in the dynamical properties of the action on Teichm�ller space
together with its Thurston boundary. Namely:

Finite-order elements fix a point inside Teichm�ller space (more concretely
this means that any mapping class of finite order in ${\displaystyle \mathrm{MCG}(S)}$
can be realized as an isometry for some hyperbolic metric on $S$);
Pseudo-Anosov classes fix the two points on the boundary corresponding
to their stable and unstable foliation and the action is minimal (has
a dense orbit) on the boundary; Reducible classes do not act minimally
on the boundary.

One can also think of $T(S)$ as a space of representations of $\pi_{1}(S)$
in the following sense:
\begin{thm}
The space $T(S)$ can be identified with a connected component in
the representation space $Rep(\pi_{1}(S),PSL(2,\mathbb{R}))=Hom(\pi_{1}(S),PSL(2,\mathbb{R}))/PSL(2,\mathbb{R})$,
where $PSL(2,\mathbb{R})$ acts by conjugation.
\end{thm}

This brought people to define higher Teichmuller spaces in the following
way:

Let $G$ be a simple Lie group $G$ of higher rank, (such as $PSL(n,\mathbb{R})$
for $n\geq3$).

A higher Teichm�ller space is a subset of $Hom(\pi_{1}(S),G)/G$ which
is a union of connected components that consist entirely of discrete
and faithful representations. 

For nice enough $G$, these higher spaces were also shown to have
proper actions of the mapping class group

We thus expect that in the arithmetic setting, the role of the Teichmuller
space will be some deformation space of Galois representations.

For deformations of the trivial representation mod $p$ this is the
same as representations of $G_{K}(p)$, since all representations
which are trivial mod $p$ factor through $G_{K}(p)$
\begin{rem}
The methods used in the proof of theorem \ref{thm: Moch anab} involve
p-adic Hodge theory, and special type of Hodge-Tate representations.

Looking at outer automorphism of $G_{K}(p)$ we also hope we could
say more about different p-adic Hodge theory structures sitting inside
the space of Galois representations.
\end{rem}

\section{\label{sec:Bass-serre}From curves to graphs of groups}

\subsection{Splittings of groups and graphs of groups}

Recall from that a simple closed curve on a surface $S$ is a continuous
injective map $S^{1}\rightarrow S$.
\begin{fact}
\label{fact:scc and split}A simple closed curve on a surface $S$,
gives rise to a splitting of $\pi_{1}(S)$ as an amalgamated free
product over $\mathbb{Z}$, or as an HNN extension over $\mathbb{Z}$.

In fact all splittings over $\mathbb{Z}$ of $\pi_{1}(S)$, into HNN
extensions and amalgamated free products, arise in such a way, and
so we have a bijection between isotopy classes of simple closed curves,
and such splittings of $\pi_{1}(S)$ over $\mathbb{Z}$.
\end{fact}

This can be found in \cite{zieschang2006surfaces} theorem 4.12.1.

The above leads us to a way to generalize ideas related to simple
closed curve to a any group.

First, let us recall some basic definitions:
\begin{defn}
(amalgamated free product) Let $A$, $B$ and $C$ be groups with
monomorphisms $\varphi_{1}:C\rightarrow A$, $\varphi_{2}:C\rightarrow B$.
We define the amalgamated free product of $A$ and $B$ over $C$,
denoted by $A*_{C}B$, to be the group generated by $A$ and $B$
together with the relation $\varphi_{1}(c)\varphi_{2}(c)^{-1}=1$
for all $c\in C$.
\end{defn}

\begin{defn}
(HNN extension) Let $A$ be a group and $C\subset A$ be a subgroup,
and let $\theta:C\rightarrow A$ be a monomorphism. We define the
HNN (G. Higman, B. H. Neumann, H. Neumann) extension of $A$ relative
to $\theta$, denoted by $A*_{\theta}$ to be the group generated
by $A$ and an element $t$ with the relation $\theta(c)=tct^{-1}$
for all $c\in C$. The new generator $t$ is called the stable letter.
\end{defn}

Using graphs of groups one can think of both of these ideas together:
\begin{defn}
A graph $X$, consists of a set of vertices $V(X)$, and a set of
(oriented) edges $E(X)$, an edge reversal map $e\mapsto\bar{e}$
such that $e\neq\bar{e}$ and $\bar{\bar{e}}=e$ for all edges (that
is each edge has an orientation, and we also have the reverse orientation
edge as part of the graph) and a map $v:E(X)\rightarrow V(X)$ sending
an edge to its terminal vertex (the initial vertex of an edge is thus
$v(\bar{e})$.
\end{defn}

\begin{defn}
A graph of groups, $(X,\boldsymbol{G})$, is a connected graph $X$
together with a function $\boldsymbol{G}$, which assigns a group
$\boldsymbol{G}_{v}$ for all vertices of the graph $V(X)$, and $\boldsymbol{G}_{e}$
for all edges $e\in E(X)$ such that $\boldsymbol{G}_{e}=\boldsymbol{G}_{\bar{e}}$,
and for each edge $e\in E(X)$ we also have a monomorphism $\varphi_{e}:\boldsymbol{G}_{e}\rightarrow\boldsymbol{G}_{v(e)}$
where $v(e)$ is the terminal vertex of $e$.

Given a graph of groups $(X,\boldsymbol{G})$, let $T\subset X$ be
a maximal subtree of $X$. We define the fundamental group $\pi_{1}(X,\boldsymbol{G},T)$
of $(X,\boldsymbol{G})$ at $T$, to be the group generated by the
vertex groups $G_{v}$ for all $v\in V(X)$ and the elements $t_{e}$
for all $e\in E(X)$ subject to the following relations:

1. $t_{\bar{e}}=t_{e}^{-1}$ for all $e\in E(X)$

2. $t_{e}\varphi_{e}(c)t_{e}^{-1}=\varphi_{\bar{e}}(c)$ for all $e\in E(X)$
and $c\in\boldsymbol{G}_{e}$ 

3. $t_{e}=1$ for all $e\in E(T)$.
\end{defn}

\begin{rem}
The above definition of fundamental group is independent of the choice
of maximal subtree $T$.
\end{rem}

We can now see that for a graph of groups with 1 edge and 1 vertex
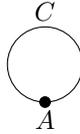
\begin{center} \begin{tikzpicture}
\filldraw  
(0,0) circle (2pt) node[align=center,   below] {$A$} ;
\draw (0,0)arc(-90:270:0.5);
\draw (0,1.2)node {$C$} ;
\end{tikzpicture}\end{center}  the fundamental group gives us an HNN extension $A*_{C}$.

For a graph of groups with 1 edge and 2 vertices\begin{center} \begin{tikzpicture}
\filldraw  
(0,0) circle (2pt) node[align=center,   below] {$A$} 
-- node[align=center, below] {$C$} 
(4,0) circle (2pt) node[align=center, below] {$B$};
\end{tikzpicture}\end{center}  the fundamental group gives us an amalgamated free products $A*_{C}B$.
\begin{defn}
Let $(X,\boldsymbol{G}),$ $(X',\boldsymbol{G}')$ be a graph of groups.
An isomorphism of $(X,\boldsymbol{G})$ to $(X',\boldsymbol{G}')$
is a tuple of the form $(H_{X},(H_{v})_{v\in V(X)},(H_{e})_{e\in E(X)},(\delta(e))_{e\in E(X)})$,
where:

\textbullet{} $H_{X}:X\rightarrow X'$ is a graph isomorphism, 

\textbullet{} $H_{v}:\boldsymbol{G}_{v}\rightarrow\boldsymbol{G}'_{H_{X}(v)}$
is a group isomorphism,

\textbullet{} $H_{e}=H_{\bar{e}}:\boldsymbol{G}_{e}\rightarrow\boldsymbol{G}'_{H_{X}(e)}$
is a group isomorphism, 

\textbullet{} $\delta(e)$ is an element of $\boldsymbol{G}'_{v(H_{X}(e))}$,
with the additional compatibility requirement that

$H_{v(e)}(\varphi_{e}(c))=\delta(e)\varphi_{H_{X}(e)}(H_{e}(c))\delta(e)^{-1}$
for all $e\in E(X),c\in\boldsymbol{G}_{e}$.
\end{defn}

Given a graph of groups $(X,\boldsymbol{G})$, one can associate the
Bass-Serre covering tree $\tilde{X}$, which is a certain tree with
a $\pi_{1}(X,\boldsymbol{G})$ action on it, without edge inversion.
This is a kind of universal cover of $(X,\boldsymbol{G})$, and one
has that $\tilde{X}/\pi_{1}(X,\boldsymbol{G})$ is isomorphic to $(X,\boldsymbol{G})$.

Similarly, given a tree $T$ for which a group $G$ acts on without
edge inversion, one has a notion of the graph of group quotient $T/G$,
for which the Bass-Serre covering is $T$.

To make this more precise, the graph $X$ of groups quotient will
be the quotient graph $T/G$, and vertex and edge groups will be isomorphic
to the vertex and edge stabilizers of vertices and edges of $T$ under
the $G$ action.

In such a way we can move between the notion of trees with a group
action on them, and graphs of groups.

One can also see that two graphs of groups are isomorphic if their
Bass-Serre trees have a $G$-equivariant graph isomorphism.
\begin{example}
Let $(X,\boldsymbol{G})$ be a graph of groups with a single edge
$e$ and two vertices $v_{1},v_{2}$, with $\boldsymbol{G}_{e}=C$,
$\boldsymbol{G}_{v_{1}}=A$,$\boldsymbol{G}_{v_{2}}=B$ with monomorphisms
$\varphi_{1}:\boldsymbol{G}_{e}\rightarrow\boldsymbol{G}_{v_{1}}$,
$\varphi_{2}:\boldsymbol{G}_{\bar{e}}\rightarrow\boldsymbol{G}_{v_{2}}$.

 \begin{center} \begin{tikzpicture}
\filldraw  
(0,0) circle (2pt) node[align=center,   below] {$A$} 
-- node[align=center, below] {$C$} 
(4,0) circle (2pt) node[align=center, below] {$B$};
\end{tikzpicture}\end{center}

We have $\pi_{1}(X,\boldsymbol{G})\cong A*_{C}B=G$.

In this case we would have $\tilde{X}$would be the tree with $V(\tilde{X})=\{gA:g\in G\}\amalg\{gB:g\in G\}$,
two vertices $gA$ and $fB$ have an edge between them if there is
$a\in A$ such that $fB=gaB$. The degree of a vertex $gA$ will be
$[A:\varphi_{1}(C)]$ and of $gB$ it will be $[B:\varphi_{2}(C)]$.
\end{example}

\begin{defn}
\label{def:Group splitting}Given a group $G$, we define its set
of splittings $Split(G)$ to be the set of isomorphism classes of
1 edge graphs of groups (i.e. HNN extensions and amalgamated free
products), with fundamental group isomorphic to $G$.

The set of $\mathbb{Z}$ splittings, $Split(G,\mathbb{Z})$, are splittings
for which all edge stabilizers are isomorphic to $\mathbb{Z}$.
\end{defn}

\begin{defn}
\label{def:Curve complex}We define the curve complex of $G$, denoted
by $\mathcal{C}(G)$, to be the complex whose vertices are $\alpha\in$$Split(G,\mathbb{Z})$,
and n vertices $v_{1},...,v_{n}$ span an n-dimensional simplex if
and only if they are pairwise non-intersecting (i.e. compatible),
in other words, if there is a graph of groups with n edges, all having
edge group $\mathbb{Z}$, such that it is a refinement of all $v_{i}$.
\end{defn}

\subsection{Intersection of Splittings}

We would also like to understand when do two simple closed curves
intersect using the splitting language.
\begin{defn}
Let $(X,\boldsymbol{G})$ be a graph of groups and let $T$ be a subgraph,
we define a new graph $(X/T,\boldsymbol{G}(X/T))$, the graph obtained
by collapsing the tree , to have vertices $V(X)\backslash V(T)\cup v_{T}$
where $v_{T}$ is a new vertex representing $T$. And the edges to
be all edges between $V(X)\backslash V(T)$ as in $X$ together with
edges between $v_{T}$ and $V(X)\backslash V(T)$ if there was an
edge joining a vertex in $V(X)\backslash V(T)$ to a vertex in $T$.

The groups for the edges and vertices of $X\backslash T$ are as in
the original $(X,\boldsymbol{G})$ and the group for the new vertex
$v_{T}$ is the fundamental group of $(T,\boldsymbol{G}|_{T})$. 
\end{defn}

\begin{example}
\label{exa:triangle}We can us the above to compute the fundamental
group of the following graph of groups:

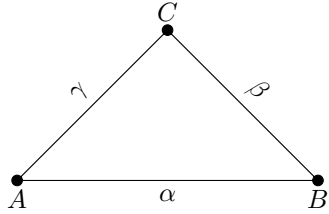
\begin{center} \begin{tikzpicture}
\filldraw 
(0,0) circle (2pt) node[align=center,   below] {$A$} 
-- node[align=center, below] {$\alpha$} 
(4,0) circle (2pt) node[align=center, below] {$B$} 
-- node[midway, above, sloped] {$\beta$} 
(2,2) circle (2pt) node[align=center,  above] {$C$}
-- node[midway, above, sloped] {$\gamma$} 
(0,0);
\end{tikzpicture}\end{center}

We start by collapsing the edge of $\alpha$ and so we get

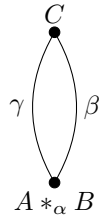
\begin{center} \begin{tikzpicture}
\filldraw(2,0) circle (2pt) node[align=center, below] {$A*_\alpha B$}; 
\draw (2,0)to[bend right] (2,2);
\draw (2.5,1) node {$\beta$};
\filldraw (2,2) circle (2pt) node[align=center,  above] {$C$};
\draw (2,2) to[bend right](2,0);
\draw (1.5,1) node {$\gamma$};
\end{tikzpicture}\end{center}

next collapsing the edge of $\beta$ we get

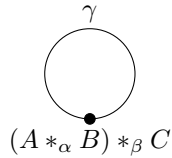
\begin{center} \begin{tikzpicture}
\filldraw 
(2,0) circle (2pt) node[align=center,  below] {$(A*_\alpha B)*_\beta C$};
\draw (2,0) arc(-90:270:0.6);
\draw (2,1.4)node {$\gamma$};
\end{tikzpicture}\end{center}

and thus the fundamental group is: $((A*_{\alpha}B)*_{\beta}C)*_{\gamma}$.
\end{example}

\begin{defn}
\label{def:intersection}We call two splittings compatible, or non-intersecting
if there is a graph of groups such that collapsing one edge (collapsing
the subgraph given by the edge and its two vertices) gives us the
first splitting and collapsing another gives us the second. We say
$i(\alpha,\beta)=0$ if they are compatible, and we say $i(\alpha,\beta)\neq0$
if they are not compatible.
\end{defn}

For example, two free products $A*B=G=A'*B'$ are compatible if and
only if there is a splitting of $G=C*D*F$ such that $(C*D)=A,F=B$
and $C=A',(D*F)=B'$. 

Non-compatible splittings are also called intersecting.

One can show (see \cite{scott2000splittings}) that two curves intersect
if and only if the corresponding splittings intersect.
\begin{example}
Suppose we have the situation as shown in the picture below:

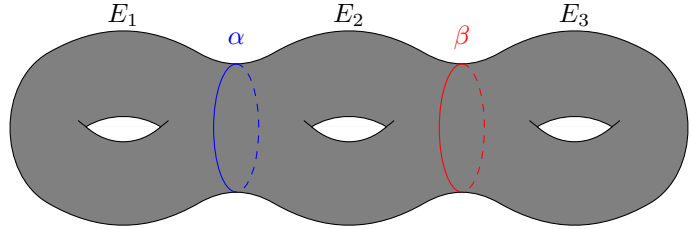
\begin{center} 
\begin{tikzpicture} 
\filldraw[fill=gray](0,1) to[out=30,in=150] (2,1) to[out=-30,in=210] (3,1) to[out=30,in=150] (5,1) to[out=-30,in=210] (6,1) to[out=30,in=150] (8,1) to[out=-30,in=30] (8,-1) to[out=210,in=-30] (6,-1) to[out=150,in=30] (5,-1) to[out=210,in=-30] (3,-1) to[out=150,in=30] (2,-1) to[out=210,in=-30] (0,-1) to[out=150,in=-150] (0,1);
\draw[smooth] (0.4,0.1) .. controls (0.8,-0.25) and (1.2,-0.25) .. (1.6,0.1);
\filldraw[fill=white][smooth] (0.5,0.02) .. controls (0.8,-0.25) and (1.2,-0.25) .. (1.5,0.02);
\filldraw[fill=white][smooth] (0.5,0.015) .. controls (0.8,0.2) and (1.2,0.2) .. (1.5,0.015);
\draw[smooth] (3.4,0.1) .. controls (3.8,-0.25) and (4.2,-0.25) .. (4.6,0.1);
\filldraw[fill=white][smooth] (3.5,0.02) .. controls (3.8,-0.25) and (4.2,-0.25) .. (4.5,0.02); 
\filldraw[fill=white][smooth] (3.5,0.015) .. controls (3.8,0.2) and (4.2,0.2) .. (4.5,0.015); 
\draw[smooth] (6.4,0.1) .. controls (6.8,-0.25) and (7.2,-0.25) .. (7.6,0.1); 
\filldraw[fill=white][smooth] (6.5,0.02) .. controls (6.8,-0.25) and (7.2,-0.25) .. (7.5,0.02);
\filldraw[fill=white][smooth] (6.5,0.015) .. controls (6.8,0.2) and (7.2,0.2) .. (7.5,0.015);
\draw [color=blue](2.5,0.85) arc(270:90:0.3 and -0.85);
\draw[color=blue][dashed] (2.5,0.85) arc(270:450:0.3 and -0.85);
\draw [color=red](5.5,0.85) arc(270:90:0.3 and -0.85);
\draw[color=red][dashed] (5.5,0.85) arc(270:450:0.3 and -0.85);
\node at (2.5,1.2) {$ \color{blue} \alpha $}; 
\node at (5.5,1.2) {$ \color{red} \beta $}; 
\node at (1,1.5) {$  E_1 $}; 
\node at (4,1.5) {$ E_2 $}; 
\node at (7,1.5) {$ E_3 $}; 
\end{tikzpicture} \end{center}

Namely a group $G=\pi_{1}(S_{3})$ and two non intersecting curves
$\alpha$ (the blue curve) and $\beta$ (the red curve). Cutting along
them we get three subsurfaces $E_{1},E_{2},E_{3}$.

The splitting of $\alpha$ will then be $\pi_{1}(E_{1})*_{\pi_{1}(\alpha)}\pi_{1}(E_{2}\cup E_{3})$
and of $\beta$ this would be $\pi_{1}(E_{1}\cup E_{2})*_{\pi_{1}(\beta)}\pi_{1}(E_{3})$.

The two curves can easily be seen to be non intersecting, and both
of these can be realized at the same time as $G=\pi_{1}(E_{1})*_{\pi_{1}(\alpha)}\pi_{1}(E_{2})*_{\pi_{1}(\beta)}\pi_{1}(E_{3})$.

In a graph of groups form this would be:
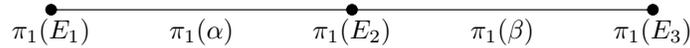
\begin{center} \begin{tikzpicture}
\filldraw  
(0,0) circle (2pt) node[align=center,   below] {$\pi_1(E_1)$} 
-- node[align=center, below] {$\pi_1(\alpha)$} 
(4,0) circle (2pt) node[align=center, below] {$\pi_1(E_2)$} 
-- node[align=center, below] {$\pi_1(\beta)$} 
(8,0) circle (2pt) node[align=center,  below] {$\pi_1(E_3)$} ; 
\end{tikzpicture}\end{center}
\end{example}

\begin{example}
Let us go back to example \ref{exa:triangle} and see what a geometric
interpretation of it could be, say on a genus 2 surface.

First note that any graph of groups coming from collapsing the one
in example \ref{exa:triangle} would of course not intersect the collapsings,
by definition.

From the form of the splitting, we get that the curves corresponding
to the edges will be non separating curves, such that cutting any
two of them would separate the surface into two components, and cutting
all three would give us 3 components.

From the above we see that a possible picture would be the following
picture:

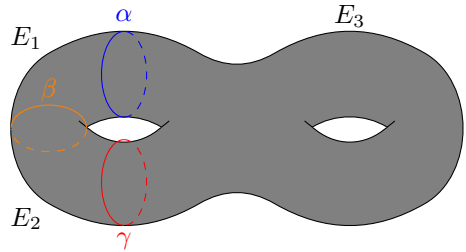
\begin{center} 
\begin{tikzpicture}
\filldraw[fill=gray] (0,1) to[out=30,in=150] (2,1) to[out=-30,in=210] (3,1) to[out=30,in=150] (5,1) to[out=-30,in=30] (5,-1) to[out=210,in=-30] (3,-1) to[out=150,in=30] (2,-1) to[out=210,in=-30] (0,-1) to[out=150,in=-150] (0,1);
\draw[smooth] (0.4,0.1) .. controls (0.8,-0.25) and (1.2,-0.25) .. (1.6,0.1);
\filldraw[fill=white][smooth] (0.5,0.02) .. controls (0.8,-0.25) and (1.2,-0.25) .. (1.5,0.02);
\filldraw[fill=white][smooth] (0.5,0.015) .. controls (0.8,0.2) and (1.2,0.2) .. (1.5,0.015);
\draw[smooth] (3.4,0.1) .. controls (3.8,-0.25) and (4.2,-0.25) .. (4.6,0.1);
\filldraw[fill=white][smooth] (3.5,0.02) .. controls (3.8,-0.25) and (4.2,-0.25) .. (4.5,0.02); 
\filldraw[fill=white][smooth] (3.5,0.015) .. controls (3.8,0.2) and (4.2,0.2) .. (4.5,0.015); 
\draw [color=blue](1.0,0.15) arc(270:90:0.3 and 1.14/2);
\draw[color=blue][dashed] (1.0,0.15) arc(270:450:0.3 and 1.14/2);
\draw [color=red](1.0,-0.15) arc(270:90:0.3 and -1.14/2);
\draw[color=red][dashed] (1.0,-0.15) arc(270:450:0.3 and -1.14/2);
\draw [color=orange](0.5,0.01) arc(180:0: -0.5 and 0.3);
\draw [color=orange][dashed](0.5,-0.01) arc(180:360: -0.5 and 0.3);
\node at (1,1.5) {$ \color{blue} \alpha $}; 
\node at (0,0.5) {$ \color{orange} \beta $}; 
\node at (1,-1.5) {$ \color{red} \gamma $}; 
\node at (-0.3,1.2) {$  E_1 $}; 
\node at (-0.3,-1.2) {$  E_2 $};
\node at (4,1.5) {$ E_3 $}; 
\end{tikzpicture} 
\end{center}

and so we see that in fact, we have cut 3 times along the same homotopy
class of a simple closed curve, which of course does not intersect
itself. 

The graph of groups can be written as follows:

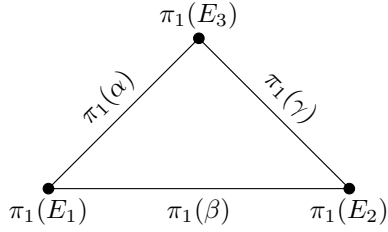
\begin{center} \begin{tikzpicture}
\filldraw 
(0,0) circle (2pt) node[align=center,   below] {$\pi_1(E_1)$} 
-- node[below] {$\pi_1(\beta)$} 
(4,0) circle (2pt) node[align=center, below] {$\pi_1(E_2)$} 
-- node[midway, above, sloped]{$\pi_1(\gamma)$} 
(2,2) circle (2pt) node[align=center,  above] {$\pi_1(E_3)$}
-- node[midway, above, sloped] {$\pi_1(\alpha)$} 
(0,0);
\end{tikzpicture}\end{center}
\end{example}

\begin{rem}
One can actually take this a step further and measure how non compatible
splittings are and define an intersection number for them.

Let $T_{1}$ and $T_{2}$ be two simplicial trees with an action of
a group $G$ on them, let $C\subset T_{1}\times T_{2}$ be the smallest
non-empty closed invariant (under the diagonal action) subset of $T_{1}\times T_{2}$
having convex fibers (a subset $E\subset T_{1}\times T_{2}$ has convex
fibers if for both $i\in\{1,2\}$ and every $x\in T_{i}$ , $E\cap p_{i}^{-1}(x)$
is convex where $p_{i}$ is the canonical projection).

This is also known as the core or the Guirardel core of the two graphs
(see \cite{guirardel2004core} for more information).

We define the intersection number $i(T_{1},T_{2})$ to be the covolume
of the core.

If the two splittings don't intersect if and only if their core is
1-dimensional.

A common refinement for them is then $\hat{C}$, the smallest connected
invariant subset of $T_{1}\times T_{2}$ having convex fibers.

Note that there is an equivalent way using almost invariant subsets
of the group $G$ associated to the splitting, the intersection number
of the splittings is then determined by intersections of certain sets
and cosets. (see \cite{scott2000splittings}).

Again it is known this is compatible with the notions of intersection
numbers of simple closed curves on a surface.

We hope to use such notations in later works to better understand
the relations between Dehn twists coming from intersecting splittings.
\end{rem}

Define the translation length function $l_{T}:G\rightarrow\mathbb{R}$,
of an action of $G$ on a tree $T$ to be:
\[
l_{T}(g)=inf_{p\in T}d_{T}(p,pg)
\]
An element $g$ is said to be elliptic with respect to the $G$-tree
$T$ if $l_{T}(g)=0$, in which case it fixes a point, and hyperbolic
if the translation length $l_{T}(g)$ is positive.

Let $\alpha=A*_{<c>}B$ (or $A*_{<c>})$ and $\beta=A'*_{<c'>}B'$
(or $A'*_{<c'>})$ be one edge $\mathbb{Z}$-splittings of a group
$G$. 

Define the translation length of $\alpha$ on $\beta$ to be:

We say that a $\beta$ is hyperbolic with respect to $\alpha$ if
the action of $c'$ on $T$, the Bass-Serre tree of $\alpha$, doesn't
fix a point. (note that by \cite{rips1995cyclic} this is true if
and only if $\beta$ is hyperbolic with respect to $\alpha$).

If this is the case, then the action of $c'$ instead preserves a
line in T, and it acts on it by translation of length $l_{T}(c')$.

So we see that $\beta$ is hyperbolic with respect to $\alpha$, if
$l_{T}(c')>0$, for $T$, the Bass-Serre tree of $\alpha$.

This will also be equivalent to having that $c'$ is not contained
in a conjugate of $A$ or $B$.

By remark 3.4 in \cite{scott2000splittings}, if $\alpha$ is hyperbolic
with respect to $\beta$ then the splittings intersect.

\subsection{Dehn twists for discrete groups}
\begin{defn}
\label{def:Dehn autom} 

Given a group and a $\mathbb{Z}$ splitting $A*_{<c>}B$, we define
the Dehn twist $\delta$ to be the automorphism $\delta(a)=a,\delta(b)=\varphi(c)b\varphi(c)^{-1}$
for all $a\in A,b\in B$.

Similarly given an HNN extension $A*_{\theta}$ where $\theta:<c>\rightarrow A$
we define the Dehn twist $\delta$ to be $\delta(a)=a$ for all $a\in A$
and $\delta(t)=ct$.
\end{defn}

\begin{example}
\label{exa:neuk outer is dehn}Let $n$ be an even number and let
$G$ be the group with the following presentation: $G=<x_{1},x_{2},...,x_{n}:x_{1}^{p}[x_{1},x_{2}]...[x_{n-3},x_{n-2}][x_{n-1},x_{n}]=1>$.
Let $A$ be the free group on $n-1$ generators $x_{1},...,x_{n-1}$,
and let $\theta$ be the map sending $x_{n-1}$ to $x_{1}^{p}[x_{1},x_{2}]...[x_{n-3},x_{n-2}]x_{n-1}$,
we thus get the splitting of $G$ as the HNN extension $A*_{\theta}=\{x_{1,...,}x_{n-1},t|tx_{n-1}t^{-1}=x_{1}^{p}[x_{1},x_{2}]...[x_{n-3},x_{n-2}]x_{n-1}\}$.
The Dehn twist associated to this splitting will act as follows: $x_{i}\mapsto x_{i}$
for all $1\leq i\leq n-1$, and $x_{n}\mapsto x_{n-1}x_{n}$.

One easily sees this is very similar to the outer automorphism constructed
in definition \ref{def: neuk outer}.

Note that for $A'=<x_{1},...,x_{n-2},x_{n}>$ and $\theta'$, we have
get another splitting of $G$ as an HNN extension $A'*_{\theta'}=\{x_{1,...,}x_{n-2},x_{n},t|tx_{n}t^{-1}=x_{1}^{p}[x_{1},x_{2}]...[x_{n-3},x_{n-2}]x_{n}\}$,
and one can see easily check they intersect.
\end{example}

\begin{defn}
\label{def:general Dehn autom}The above definition of a Dehn twist
is the same as the automorphism of a graph of groups $(X,\boldsymbol{G})$
given by $H_{X}=id_{X}$,$H_{v}=id_{\boldsymbol{G}_{v}}$, $H_{e}=id_{\boldsymbol{G}_{e}}$
and $\delta$ is such that there is an element $\gamma_{e}$ in the
center $Z(\boldsymbol{G}_{e})$ for each edge group such that $\delta(e)=\varphi_{e}(\gamma_{e})$.
A collection $(\gamma_{e})_{e\in E(X)}$ with each $\gamma_{e}\in Z(\boldsymbol{G}_{e})$
and a base point $v\in V(X)$ defines a Dehn twist.

In other words, take $\gamma_{e}\in Z(\boldsymbol{G}_{e})$ for each
$e\in E(X)$ and choose a base point $x_{0}$ of $X$. Then the Dehn
twist of $(\gamma_{e},x_{0})$ will act on $\pi_{1}(X,\boldsymbol{G},T)$
as follows:
\[
\begin{array}{c}
t_{e}\rightarrow\gamma_{e}t_{e}\\
g_{v}\rightarrow\gamma_{e_{1}}...\gamma_{e_{n}}g_{v}\gamma_{e_{n}}^{-1}...\gamma_{e_{1}}^{-1}
\end{array}
\]
 for all $t_{e}$ for which $e\notin T$, and for all $g_{v}\in\boldsymbol{G}_{v}$
where $e_{1}..e_{n}$ the path in $T$ from $x_{0}$ to $v$.
\end{defn}

\begin{rem}
\label{rem:compatible collapse}One gets from the definition above
that: 

-The action on the group of the base point is trivial.

-Changing base point from $x_{0}$ to $x_{1}$ will change the Dehn
twist by conjugation by $\gamma_{f_{1}}...\gamma_{f_{n}}$ for a path
$f_{1}...f_{n}$ in $T$ between $x_{0}$ and $x_{1}$.

-The definition is compatible with definition \ref{def:Dehn autom}.
For an amalgamated product, the choice of base point will be the one
with the group $A$.

-Let $T'$ be a subtree of $X$, we will get that Dehn twists for
$(X/T',\boldsymbol{G}(X/T')$ are the same as Dehn twists for $(X,\boldsymbol{G})$
taking $\gamma_{t}=1$ for all $t\in T'$.

Later when we look at graphs of groups with cyclic edge groups, the
collection $(\gamma_{e})$ will be the edge group generators.
\end{rem}

\begin{rem}
Once again this is all compatible and equal to the Dehn twists coming
from simple closed curves.
\end{rem}

Given all of this one can ask when do Dehn twists give us non trivial
mapping classes, or non trivial outer automorphisms of $G$.

We finish the subsection with two lemmas about intersections of splittings.
\begin{lem}
\label{lem:Conj intersection}Conjugation doesn't change whether two
splittings intersect or not
\end{lem}

\begin{proof}
See \cite{scott2000splittings} remark 1.9.
\end{proof}
\begin{lem}
\label{lem:no inters commute}If two splittings $\alpha,\beta$ don't
intersect, then their Dehn twists commute
\end{lem}

\begin{proof}
Since the the splittings don't intersect they have a common refinement.
Then their Dehn twists can be realized on their common refinement,
as in definition \ref{def:general Dehn autom}, looking at that definition,
we can see that the Dehn twists must commute (see remark \ref{rem:compatible collapse}).
\end{proof}

\subsection{\label{subsec:Profinite Bass-serre}Profinite groups and graphs}
\begin{defn}
\label{def:profin. graphs}A topological space $T$ is a profinite
space if it is the inverse limit of finite discrete spaces, and so
$T$ is profinite if it is compact, Hausdorff and totally disconnected.
\end{defn}

\begin{defn}
Given a discrete group $G$, one can define its profinite completion
$\hat{G}$ to be the limit of $G/N$ over all finite index normal
subgroup $N\vartriangleleft G$ and pro-p completion to be the limit
of $G/N$ over all normal subgroups $N$ with $|G/N|=p^{r}$ for some
integer $r$.
\end{defn}

\begin{rem}
The above actually gives a functor from the category of discrete groups
to that of profinite/pro-p groups. This completion Functor is right
exact.
\end{rem}

\begin{defn}
We call a group $G$ residually finite (resp. residually p), if for
every element $1\neq g\in G$, there is a homomorphism $\varphi:G\rightarrow H$
to a finite group (resp. p-group), such that $\varphi(g)\neq1$.
\end{defn}

\begin{rem}
A group is residually finite (resp. residually p) if and only if the
natural map to its profinite (resp. pro-p) completion is injective.
\end{rem}

\begin{defn}
A topological graph $X$ is a topological space $X$, together with
a subset of vertices $V(X)$ and two incidence maps $v_{0},v_{1}:X\rightarrow V(X)$
which are the identity on $V(X)$ (these are the maps sending an edge
to its terminal and initial vertices). It is called a profinite graph
if $X$ is a profinite space, the subset of vertices $V(X)$ is closed
(and hence profinite) and the two incidence maps $v_{0},v_{1}:X\rightarrow V(X)$
are continuous.
\end{defn}

If $E(X)=X\backslash V(X)$ is closed, it is enough to define $v_{0}$
and $v_{1}$ continuously on $E(X)$. In general, the set of edges
$E(X)$ of a profinite graph $X$ need not be a closed (and therefore
compact) subset of $X$.

Finite abstract graphs are profinite graphs. A morphism of profinite
graphs is a morphism of abstract graphs which is also continuous. 
\begin{fact}
Every profinite graph $X$ can be represented as an inverse limit
of finite quotient graphs of $X$.
\end{fact}

Using the above fact of representing a profinite graph as an inverse
limit of finite graphs, one can define the free profinite $\mathbb{F}_{p}$
on the space, and then similarly to the abstract situation, one can
define a chain complex and homology groups of $X$ with coefficients
in $\mathbb{F}_{p}$.

A profinite graph $X$ is said to be a pro-p tree if $X$ is connected
(if and only if $H_{0}(X,\mathbb{F}_{p})=0)$ and $H_{1}(X,\mathbb{F}_{p})=0$.
Thus $X$ is a pro-p tree if and only if $C(X,\mathbb{F}_{p})$ is
an exact sequence.
\begin{fact}
\label{fact:pro-p trees facts}(a) Every connected subgraph of a pro-p
tree is a pro-p tree. 

(b) An inverse limit of pro-p trees is a pro-p tree.

(c) Let $X_{i},i\in I$, be a family of pro-p subtrees of a pro-p
tree $T$. Then the subgraph $X=\cap_{i}X_{i}$ is a pro-p tree (possibly
empty).
\end{fact}

Let $G$ be a pro-p group and $X$ a profinite graph. Suppose that
$G$ acts continuously on the profinite space $X$ from the left. 

We say that $G$ acts on the profinite graph $X$ (from the left)
if $v_{i}(gx)=gv_{i}(x)$ for all $g\in G$,$x\in X$ and $i=0,1$. 

Define $G_{x}=\{g\in G:gx=x\}$ to be the stabilizer of an element
$x\in X$.

Let $\varphi_{X}:G\times X\rightarrow X$ be the map defined by $\varphi_{X}(g,x)=gx$
and $pr_{G}:G\times X\rightarrow G$ be the projection. 

Then $G_{x}=pr_{G}(\varphi_{X}^{-1}(x)\cap(G\times\{x\}))$, and so
$G_{x}$ is a closed subgroup of $G$.

If the stabilizer $G_{x}$ of every element $x\in X$ is trivial,
we say that $G$ acts freely on $X$.

If a pro-p group $G$ acts on a profinite graph $R$, then the space
$G\backslash X=\{Gx:x\in X\}$ of $G$-orbits with the quotient topology
is a profinite space which admits a natural profinite graph structure
as follows:

$V(G\backslash X)=G\backslash V(X),v_{i}(Gx)=Gv_{i}(x)$ for $i=0,1$.

So one has a natural epimorphism $X\rightarrow G\backslash X$ of
profinite graphs.

If $N\vartriangleleft G$, there is an action of $G/N$ on $N\backslash X$
defined in the standard way: $(gN)(Nx)=Ngx$ for $g\in G$, $x\in X$.

Let a pro-p group $G$ act on a profinite graph $X$ and let $N\vartriangleleft G$,
consider the natural action of $G/N$ on $N\backslash X$ defined
above. Then the stabilizers in $G/N$ are epimorphic images of the
corresponding stabilizers in $G$.
\begin{example}
\label{exa:pro-p cayley graph}Let $G$ be a pro-p group and $S$
a closed subset of $G$ containing $1$. 

Define the Cayley graph $X(G,S$) as follows: $X(G,S)=G\times S$,
$V(X(G,S))=G\times\{1\}$, $v_{0}(g,s)=(g,1)$, $v_{1}(g,s)=(gs,1)$
for all $g\in G$,$s\in S$. 

Then $X(G,S)$ is a profinite graph. Define a left action of $G$
on $X(G,S)$ by setting $g'\cdot(g,s)=(g'g,s)$ for all $(g,s)\in G\times S,g'\in G$.

Clearly $gv_{i}(x)=v_{i}(g(x))$ for all $g\in G,x\in X(G,S)$ and
$i=0,1$, so $G$ acts freely on the Cayley graph $X(G,S)$.
\end{example}

\begin{defn}
\label{def:pro-p amalgamted}Let $A$, $B$ and $C$ be pro-p groups
and $\varphi_{1}:C\rightarrow A$, $\varphi_{2}:C\rightarrow B$,
be monomorphisms of pro-p groups.

The amalgamated free pro-p product of $A$ and $B$ with amalgamated
subgroup $C$ is a pushout in the category of pro-p groups 
\[
\xymatrix{C\ar[d]_{\varphi_{2}}\ar[r]^{\varphi_{1}} & A\ar[d]_{f_{1}}\\
B\ar[r]^{f_{1}} & A*_{C}B
}
\]
 In other words, a pro-p group $G=A*_{C}B$, satisfying the following
universal property: 

For any pair of homomorphisms $\psi_{1}:A\rightarrow K$,$\psi_{2}:B\rightarrow K$
into a pro-p group $K$ with $\psi_{1}\varphi_{1}=\psi_{2}\varphi_{2}$,
there exists a unique homomorphism $\psi:G\rightarrow K$ such that
the following diagram is commutative:
\[
\xymatrix{C\ar[d]_{\varphi_{2}}\ar[r]^{\varphi_{1}} & G_{1}\ar[d]_{f_{1}}\ar@/^{1pc}/[ddr]^{\psi_{1}}\\
G_{2}\ar[r]^{f_{2}}\ar@/_{1pc}/[drr]_{\psi_{2}} & G\ar@{-->}[dr]^{\psi}\\
 &  & K
}
\]

It follows from the existence and uniqueness of push outs that an
amalgamated free pro-p product exists and it is unique.
\end{defn}

\begin{defn}
\label{def: pro-p HNN}Let $A$ be a pro-p group and $\theta:C'\rightarrow C$
a continuous isomorphism from a pro-p group to a closed subgroups
of $A$, the pro-p HNN extension of $A$ with respect to $(\theta,C)$,
is a pro-p group $G=HNN(A,\theta,C)=A*_{\theta}$, with an element
$t\in G$ and a homomorphism $\varphi:A\rightarrow G$, with the following
universal property: 

For any pro-p group $K$, any $k\in K$ and any homomorphism $\psi:A\rightarrow K$
satisfying $k\psi((c))k^{-1}=\psi(\theta(c))$ for all $c\in C'$,
there is a unique homomorphism $\omega:G\rightarrow K$ such that
the diagram
\[
\xymatrix{G\ar@{-->}[dr]^{\omega}\\
H\ar[u]^{\varphi}\ar[r]^{\psi} & K
}
\]
 is commutative and $\omega(t)=k$.
\end{defn}

\begin{rem}
Note that the above universal properties are the same as the ones
for discrete HNN and amalgamated products, but now in the category
of pro-p groups.

Also note that it is enough to check the universal property just for
finite p-groups $K$.
\end{rem}

\begin{defn}
We again define a Dehn twist of HNN extensions and amalgamated free
products, but now for pro-p/profinite versions:

Given a group and a $\mathbb{Z}_{p}$ splitting $A*_{<c>}B$, we define
the Dehn twist $\delta$ to be the automorphism $\delta(a)=a,\delta(b)=\varphi(c)b\varphi(c)^{-1}$
for all $a\in A,b\in B$. This extends to $A*_{<c>}B$ by the universal
property

Similarly given an HNN extension $A*_{\theta}$ where $\theta:<c>\rightarrow A$
we define the Dehn twist $\delta$ to be $\delta(a)=a$ for all $a\in A$
and $\delta(t)=ct$. Again this extends to $A*_{\theta}$ by the universal
property.
\end{defn}

\begin{defn}
More generally, let $(X,\boldsymbol{G})$ be a finite graph of profinite
groups, and let $T\subset X$ be a maximal subtree X. 

The profinite fundamental group of the graph of groups $(X,\boldsymbol{G})$
with respect to $T$, is a profinite group $\pi_{1}(X,\boldsymbol{G},T)$,
and a map
\[
\psi:\coprod_{v\in V(X)}\boldsymbol{G}_{v}\amalg\coprod_{e\in E(X)}<t_{e}>\rightarrow\pi_{1}(X,\boldsymbol{G},T)
\]
 such that $\varphi(t_{e})=1$ for all $e\in E(T)$ and $\psi(t_{e}\varphi_{e}(c)t_{e}^{-1})=\psi(\varphi_{\bar{e}}(c))$
for all $e\in E(X)$ and $c\in\boldsymbol{G}_{e}$, and such that
$(\pi_{1}(X,\boldsymbol{G},T),\psi)$ are universal with respect to
these properties.
\end{defn}

The above exists and is independent of the maximal subtree $T$, and
in fact is equal to 
\[
((\coprod_{v\in V(X)}\boldsymbol{G}_{v})\amalg F)/N
\]
where $F$ is the free profinite group generated by $t_{e}$ for all
$e\in E(X)$ and $N$ is the smallest closed normal subgroup of $\coprod_{v\in V(X)}\boldsymbol{G}_{v})\amalg F$
containing $\{t_{e}:e\in E(T)\}\cup\{\varphi_{\bar{e}}(c)^{-1}t_{e}\varphi_{e}(c)t_{e}^{-1}:e\in E(X),c\in\boldsymbol{G}_{e}\}$.
This all also works similarly for finite graphs of pro-p groups (see
\cite{ribes2017profinite} Section 6). 

Similarly to classical Bass-Serre theory, we have the following (this
can also be found in \cite{ribes2017profinite} Section 6):
\begin{thm}
Let $(X,\boldsymbol{G})$ be a finite graph of profinite groups. Then
there exists a unique profinite tree $S((X,\boldsymbol{G}))$, called
the standard graph of $(X,\boldsymbol{G})$, on which $\pi_{1}(X,\boldsymbol{G})$
acts, such that the quotient graph $\pi_{1}((X,\boldsymbol{G}))\backslash S((X,\boldsymbol{G}))$
is isomorphic to $X$, and such that the stabilizer of a point $s\in S(X,\boldsymbol{G})$,
is a conjugate of $Im(\boldsymbol{G}_{p(s)}\rightarrow\pi_{1}(X,\boldsymbol{G}))$,
where $p:S((X,\boldsymbol{G}))\rightarrow X$ is the projection map.

Conversely an action of a profinite group on a profinite tree for
which the quotient is a finite graph, gives rise to a decomposition
of the profinite group as a finite graph of profinite groups.
\end{thm}

\begin{rem}
Note that one has fundamental groups for graphs of pro-p groups, depends
on a universal property, coming from $\boldsymbol{G}(v)$ and $t_{e}$.
So a similar definition for a Dehn twists of graphs of pro-p groups
can be defined (similar to \ref{def:general Dehn autom}). Also note
that in general, it is not the case that it is just a pro-p completion
of a discrete Dehn twist.
\end{rem}

\section{\label{sec:Main results}Demuskin groups}

\subsection{Background on Poincare duality groups}

We now turn our focus to Demuskin groups, which are part of a larger
class of pro-p groups, which satisfy a version of Poincare duality.
\begin{defn}
Let $G$ be an infinite pro-p group.

Then $G$ is a Poincare duality group of dimension $n$ if:

- $dim_{\mathbb{F}_{p}}H^{i}(G,\mathbb{F}_{p})<\infty$ for all $i\leq n$

- $dim_{\mathbb{F}_{p}}H^{n}(G,\mathbb{F}_{p})=1$

- for all $\text{0\ensuremath{\leq i\leq n}}$ the cup product $H^{i}(G,\mathbb{F}_{p})\times H^{n-i}(G,\mathbb{F}_{p})\rightarrow H^{n}(G,\mathbb{F}_{p})$
is a non-degenerate bilinear form. 
\end{defn}

These above conditions are equivalent to a more ``standard'' Poincare
duality as follows:
\begin{defn}
Let $G$ be a profinite group of cohomological dimension $n$, such
that for every $G$-module $A$ which is a finite group, $H^{n}(G,A)$
is finite.

Define the dualizing module to be the torsion group $G$-module $I$,
which represents the functor $A\mapsto H^{n}(G,A)^{*}$
\end{defn}

We have the following proposition, which can be found in \cite{serre1979galois}:
\begin{prop}
Let $G$ be an $n$-dimensional Poincare pro-$p$ group, then it has
a dualizing module $I$, and we have:

(a) $I$ is isomorphic to $\mathbb{Q}_{p}/\mathbb{Z}_{p}$ as an abelian
group. 

(b) The canonical homomorphism $i:H^{n}(G,I)\rightarrow\mathbb{Q}/\mathbb{Z}$
is an isomorphism of $H^{n}(G,I)$ with $\mathbb{Q}_{p}/\mathbb{Z}_{p}$
(as a subgroup of $\mathbb{Q}/\mathbb{Z}$). 

(c) For all finite $G$-modules $A$ and for all integers $i$, the
cup-product 
\[
H^{i}(G,A)\times H^{n-i}(G,Hom(A,I))\rightarrow H^{n}(G,I)=\mathbb{Q}_{p}/\mathbb{Z}_{p}
\]
 gives a duality between the finite groups $H^{i}(G,A)$ and $H^{n-i}(G,Hom(A,I))$. 
\end{prop}

Denote by $U_{p}$ the group of $p$-adic units. We have that $U_{p}$
is the automorphism group of $I$. Since $G$ acts on $I,$ we have
that this action is given by a canonical homomorphism
\[
\chi:G\rightarrow U_{p}
\]
 This homomorphism is continuous; it determines $I$ (up to isomorphism),
we call $\chi$ the orientation character of $G$. This character
is an important invariant of a Poincare duality group.
\begin{example}
If $G$ is a $p$-adic analytic group of dimension $n$, which is
compact and torsion-free, then $G$ is a Poincare group of dimension
$n$ (see \cite{lazard1965groupes}). 
\end{example}

Let $G$ be a Poincare pro-p group of dimension $n>0$.

Every open subgroup of $G$ is also a Poincare duality group of dimension
$n$.

Let $H$ be a proper closed subgroup of $G$, then $Res:H^{n}(G,\mathbb{F}_{p})\rightarrow H^{n}(H,\mathbb{F}_{p})$
is 0. 

If we also have that $[G:H]=\infty$ (i.e. that $H$ is not open),
then we get that the cohomological dimension of $H$ is $\leq n$.
In particular, every closed subgroup of infinite index of a Demuskin
group is a free pro-p-group. 

Let $G$ be a Demuskin group and let $H$ be an open subgroup of $G$.
One has: 
\[
rank(H)-2=[G:H](rank(G)-2).
\]

\subsection{Main results}

We now turn to Poincare duality groups of dimension 2, which are also
known as Demuskin groups.

Such groups, are in fact one relator groups, which are actually fully
determined by two invariants associated to a Demuskin group: the minimal
number of (topological) generators, and an extra invariant, which
depends on the orientation character, which is either $\infty$ or
a power of the prime $p$. To be more precise, this invariant is the
highest power $p^{r}$, such that the induced map $\chi:G\rightarrow Aut(\mathbb{Z}_{p}/p^{r}\mathbb{Z}_{p})$
is trivial.
\begin{thm}
\label{thm:Demushkin relations }(\cite{demushkin1961group}) Let
$G$ be a Demushkin group with $d'$ generators and orientation character
which vanishes mod $p^{r}$ but not mod $p^{r+1}$, and suppose that
$p\neq2$.

Then $d'$ is even, $d'=2d$, and $G$ is isomorphic to $F/R$, where
$F$ is a free pro-p group with basis $x_{1,}y_{1},...,x_{d},y_{d}$
and $R$ is generated as a normal closed subgroup by $x_{1}^{p^{r}}[x_{1},y_{1}]...[x_{d},y_{d}]$.

For $r=\infty$ (i.e. trivial orientation character), we replace $x_{1}^{\infty}$
by the identity (and so the relation is just the usual surface group
relation $[x_{1},y_{1}]...[x_{d},y_{d}]=1$).

Furthermore all groups having such presentations are Demushkin.
\end{thm}

In the case when $G$ is a Demushkin group with $p=2$, the classification
was completed by Serre \cite{serre1962structure} and Labute \cite{labute1967classification}.

Our interest in Demuskin groups comes from two reasons, for discrete
groups, the only Poincare duality groups of dimension 2 are surface
groups.

The other and more important reason is the following classification
of the maximal pro-p quotient of the absolute Galois group of a p-adic
field (or the Galois group of the maximal p-extension);
\begin{thm}
Let $K$ be a p-adic field, of degree $N$ over $\mathbb{Q}_{p}$,
and suppose $p\neq2$:

1) If $\mu_{p}\nsubseteq K$, then $G_{K}(p)$ is a free pro-p group
of rank $N+1$

2)If $\mu_{p}\subseteq K$, let $r$ be the maximal power such that
$\mu_{p^{r}}\subseteq K$ (i.e. $\mu_{p^{r+1}}\nsubseteq K$), then
$G_{K}(p)$ is a Demuskin group on $N+2$ generators and invariant
$p^{r}$ (note that $N$ must be even since $\mu_{p}\subseteq K$),
and dualizing module, the group of all roots of unity in the maximal
p-extension of $K$.

In other words, $G_{K}(p)$ is the pro-p group on the generators $x_{1},y_{1},...,x_{\frac{N}{2}+1},y_{\frac{N}{2}+1}$
and relation $x_{1}^{p^{r}}[x_{1},y_{1}]...[x_{\frac{N}{2}+1},y_{\frac{N}{2}+1}]$.
\end{thm}

For the rest of the chapter we will, assume that $p\neq2$, and denote
by $G$ a Demushkin group, and by $\mathcal{G}_{r'}$ for $r'\geq r$
to be the discrete groups with $2d$ generators and 1 relation $x_{1}^{q}[x_{1},y_{1}]...[x_{d},y_{d}]y_{d}^{-p^{r'}}$.
When $r'=\infty$ we simply denote it by $\mathcal{G}$.

\subsubsection{Splittings of discrete Demuskin groups}
\begin{defn}
\label{def:D split and aut}

Define $Split_{D}(\mathcal{G}_{r'},\mathbb{Z})$ to be the union $Split_{D,amalg}(\mathcal{G}_{r'},\mathbb{Z})$,
the set of splittings of $\mathcal{G}_{r'}$ into $F_{2n}*_{<c>}F_{d-2n}$
where $c$ maps to $x_{1}^{q}[x_{1},y_{1}]...[x_{n},y_{n}]$ in $F_{2n}$
and to $y_{d}^{p^{r'}}[y_{d},x_{d}]..[y_{n+1},x_{n+1}]$ in $F_{d-2n}$
for some basis $x_{i},y_{i}$, and $Split_{D,HNN}(\mathcal{G}_{r'},\mathbb{Z})$,
the set of splittings into $A*_{\theta}$ where $A$ be the free group
on $n-1$ generators $x_{1},...,x_{d-1},y_{d-1},y_{d}$, and let $\theta$
be the map sending $y_{d}$ to $y_{d}^{-p^{r'}}x_{1}^{p^{r}}[x_{1},y_{1}]...[x_{d-1},y_{d-1}]y_{d}$.

We denote by $Aut_{D}(\mathcal{G}_{r'})$ the subgroup of $Aut(\mathcal{G}_{r'})$
generated by Dehn twists $T_{\alpha}$, for $\alpha\in Split_{D}(\mathcal{G}_{r'},\mathbb{Z})$,
the image of $Aut_{D}(\mathcal{G}_{r'})$ in $Out(\mathcal{G}_{r'})$
will be denoted by $Out_{D}(\mathcal{G}_{r'})$.
\end{defn}

\begin{rem}
We restrict ourselves to such splittings, since we do not know how
to classify all splittings of the discrete group, and in a sense we
don't need to for our applications.

For Demushkin groups, these are the only splittings, by \cite{wilkes2019relative}
theorem 5.18 together with \cite{wilkes2020classification} theorem
3.3, and so when we will look at Dehn twists of the pro-p case, the
only restriction we have is that the edge stabilizers are cyclic.
\end{rem}

\begin{thm}
\label{thm:HNN discrete Demuskin}Fix a basis $x_{1},y_{1}...,x_{d},y_{d}$,
let $\alpha$ be the splitting of $\mathcal{G}$ into $A*_{\theta}$
where $A$ be the free group on $d-1$ generators $x_{1},y_{1}...,x_{d}$,
and let $\theta$ be the map sending $x_{d}$ to $x_{1}^{p}[x_{1},y_{1}]...[x_{d-1},y_{d-1}]x_{d}$
then $T_{\alpha}^{k}$ is non trivial in $Out(\mathcal{G})$.
\end{thm}

\begin{proof}
Let $\beta$ be the splitting $B=<x_{1},y_{1}...x_{d-1},y_{d-1},y_{d}>$
and $\theta'$ we have $B*_{\theta'}=\{x_{1},y_{1}...x_{d-1},y_{d-1},y_{d},t|tx_{n}t^{-1}=x_{1}^{p}[x_{1},y_{1}]...[x_{d-1},y_{d-1}]y_{d}\}$.

Looking at the translation length of the action of $x_{d}$ on the
tree of $\beta$ we see that it is positive, since $x_{d}$ is not
conjugate to any element of $B$, hence the splittings $\alpha$ and
$\beta$ intersect. 

Again we have that the translation length of the action of $T_{\alpha}^{k}(y_{d})$
on the tree of $\beta$ is positive, since $y_{d}x_{d}^{k}$ is not
conjugate to any element of $B$ and so the splittings $T_{\alpha}^{k}(\beta)$
and $\beta$ intersect.

Now assume that $T_{\alpha}^{k}$ was inner, i.e. acted by conjugation.

Since conjugation does not change whether two splittings intersect
or not (lemma \ref{lem:Conj intersection}), and any splitting does
not intersect itself, we get a contradiction.

Hence $T_{\alpha}^{k}$ is outer.
\end{proof}
\begin{rem}
In the above, when $r=\infty$ we get the following picture for $\alpha$
and $\beta$:
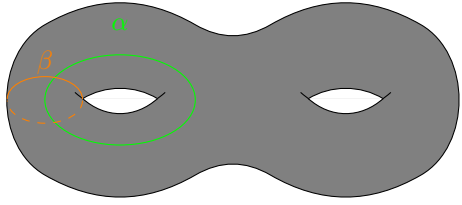
\begin{center} 
\begin{tikzpicture}
\filldraw[fill=gray] (0,1) to[out=30,in=150] (2,1) to[out=-30,in=210] (3,1) to[out=30,in=150] (5,1) to[out=-30,in=30] (5,-1) to[out=210,in=-30] (3,-1) to[out=150,in=30] (2,-1) to[out=210,in=-30] (0,-1) to[out=150,in=-150] (0,1);
\draw[smooth] (0.4,0.1) .. controls (0.8,-0.25) and (1.2,-0.25) .. (1.6,0.1);
\filldraw[fill=white][smooth] (0.5,0.02) .. controls (0.8,-0.25) and (1.2,-0.25) .. (1.5,0.02);
\filldraw[fill=white][smooth] (0.5,0.015) .. controls (0.8,0.2) and (1.2,0.2) .. (1.5,0.015);
\draw[smooth] (3.4,0.1) .. controls (3.8,-0.25) and (4.2,-0.25) .. (4.6,0.1);
\filldraw[fill=white][smooth] (3.5,0.02) .. controls (3.8,-0.25) and (4.2,-0.25) .. (4.5,0.02); 
\filldraw[fill=white][smooth] (3.5,0.015) .. controls (3.8,0.2) and (4.2,0.2) .. (4.5,0.015); 
\draw [color=green](1,0) circle (1 and 0.6);
\draw [color=orange](0.5,0.01) arc(180:0: -0.5 and 0.3);
\draw [color=orange][dashed](0.5,-0.01) arc(180:360: -0.5 and 0.3);
\node at (1,1) {$ \color{green} \alpha $}; 
\node at (0,0.5) {$ \color{orange} \beta $}; 
\end{tikzpicture} 
\end{center}
\end{rem}

\begin{thm}
\label{thm:amalgam discrete demuskin}Let $\alpha$ be the splitting
of $\mathcal{G}_{r'}$ into $F_{2n}*_{<c>}F_{2d-2n}$ where $c$ maps
to $x_{1}^{q}[x_{1},y_{1}]...[x_{n},y_{n}]$ in $F_{2n}$ and to $y_{d}^{p^{r'}}[y_{d},x_{d}]..[y_{n+1},x_{n+1}]$
in $F_{2d-2n}$ then $T_{\alpha}^{k}$ is non trivial in $Out(\mathcal{G}_{r'})$.
\end{thm}

\begin{proof}
Similarly to theorem 35:

Let $\beta$ be the splitting $B=<x_{1},y_{1},...,x_{n},y_{n+1},x_{n+2},y_{n+2},..,x_{d},y_{d},b>$
where $b=y_{n}^{-1}x_{n+1}$ with gluing map 
\[
\theta'(b^{-1}x_{n}^{-1}by_{n+1})=x_{n}^{-1}[y_{n-1},x_{n-1}]...[y_{1},x_{1}]x_{1}^{q}y_{d}^{p^{r'}}[y_{d},x_{d}]...[y_{n+2},x_{n+2}]y_{n+1}
\]
 and so we have 

\[
B*_{\theta'}=\{B,t|\begin{array}{c}
tb^{-1}x_{n}^{-1}by_{n+1}t^{-1}=\\
x_{n}^{-1}[y_{n-1},x_{n-1}]...[y_{1},x_{1}]x_{1}^{q}y_{d}^{p^{r'}}[y_{d},x_{d}]...[y_{n+2},x_{n+2}]y_{n+1}
\end{array}\}
\]
where $t$ would be $x_{n+1}$ in $\mathcal{G}_{r'}$.

Looking at the translation length of the action of $x_{n+1}^{-1}y_{n}x_{n}^{-1}y_{n}^{-1}x_{n+1}y_{n+1}$
on the tree of $\alpha$ we see that it is positive, since it is not
conjugate to any element of $<x_{1},...,y_{n}>$ or of $<x_{n+1},...,y_{d}>$
(due to $x_{n}$ and $y_{n+1}$ both having sum of powers non zero),
hence the splittings $\alpha$ and $\beta$ intersect. 

Again we have that the translation length of the action of $T_{\alpha}^{k}(b^{-1}x_{n}^{-1}by_{n+1})$
on the tree of $\beta$ is positive, since $x_{n+1}^{-1}\gamma^{k}y_{n}x_{n}^{-1}y_{n}^{-1}\gamma^{-k}x_{n+1}y_{n+1}$
is not conjugate to any element of $B$, where $\gamma=x_{1}^{q}[x_{1},y_{1}]...[x_{n},y_{n}]$
(due to the appearances of $y_{n}$ ``away from $x_{n+1}$'' in
the middle of the word, coming from $\gamma$). 

For example when $k=1$ we have: 
\[
x_{n+1}^{-1}x_{1}^{q}[x_{1},y_{1}]..[x_{n-1},y_{n-1}][x_{n},y_{n}]x_{n}^{-1}[y_{n-1},x_{n-1}]...[y_{1},x_{1}]x_{1}^{-q}x_{n+1}y_{n+1}
\]
 which is clearly not conjugate to elements of $B$, and so the splittings
$T_{\alpha}^{k}(\beta)$ and $\beta$ intersect.

Note that one can also see that $T_{\alpha}^{k}(\beta)$ and $\beta$
intersect since if they wouldn't then we would get that $T_{T_{\alpha}^{k}(\beta)}$
and $T_{\beta}$ commute (by lemma \ref{lem:no inters commute}).

But by lemma \ref{lem:cong formula} we have $T_{T_{\alpha}^{k}(\beta)}=T_{\alpha}^{k}T_{\beta}T_{\alpha}^{-k}$,
and a direct computation (say using Dehn's algorithm for the word
problem) would show that $T_{\alpha}^{k}T_{\beta}T_{\alpha}^{-k}T_{\beta}(x_{n+1})\neq T_{\beta}T_{\alpha}^{k}T_{\beta}T_{\alpha}^{-k}(x_{n+1})$.

Now assume that $T_{\alpha}^{k}$ was inner, i.e. acted by conjugation.

Since conjugation does not change whether two splittings intersect
or not (lemma \ref{lem:Conj intersection}), and any splitting does
not intersect itself, we get a contradiction.

Hence $T_{\alpha}^{k}$ is outer.
\end{proof}
\begin{rem}
Note that for having amalgamated free products we need $d>1$, similar
to how we need the genus of a surface to be great than 1 to have separating
curves.
\end{rem}

\begin{rem}
In the above, when $r=\infty$ we get something almost like the following
picture for $\alpha$ and $\beta$ (only almost like it, since $\beta$
got some extra ``twisting'' in its presentation to make it easier
to work with):

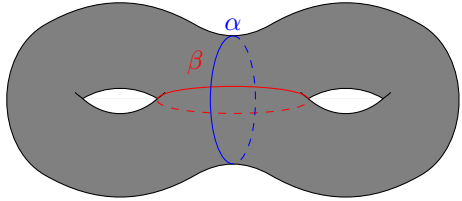
\begin{center} 
\begin{tikzpicture}
\filldraw[fill=gray] (0,1) to[out=30,in=150] (2,1) to[out=-30,in=210] (3,1) to[out=30,in=150] (5,1) to[out=-30,in=30] (5,-1) to[out=210,in=-30] (3,-1) to[out=150,in=30] (2,-1) to[out=210,in=-30] (0,-1) to[out=150,in=-150] (0,1);
\draw[smooth] (0.4,0.1) .. controls (0.8,-0.25) and (1.2,-0.25) .. (1.6,0.1);
\filldraw[fill=white][smooth] (0.5,0.02) .. controls (0.8,-0.25) and (1.2,-0.25) .. (1.5,0.02);
\filldraw[fill=white][smooth] (0.5,0.015) .. controls (0.8,0.2) and (1.2,0.2) .. (1.5,0.015);
\draw[smooth] (3.4,0.1) .. controls (3.8,-0.25) and (4.2,-0.25) .. (4.6,0.1);
\filldraw[fill=white][smooth] (3.5,0.02) .. controls (3.8,-0.25) and (4.2,-0.25) .. (4.5,0.02); 
\filldraw[fill=white][smooth] (3.5,0.015) .. controls (3.8,0.2) and (4.2,0.2) .. (4.5,0.015); 
\draw [color=red](1.5,0.01) arc(0:180: -1 and 0.17);
\draw [color=red][dashed](1.5,-0.01) arc(0:-180: -1 and 0.17);
\draw [color=blue](2.5,0.85) arc(270:90:0.3 and -0.85);
\draw[color=blue][dashed] (2.5,0.85) arc(270:450:0.3 and -0.85);
\node at (2.5,1) {$ \color{blue} \alpha $}; 
\node at (2,0.5) {$ \color{red} \beta $}; 
\end{tikzpicture} 
\end{center}
\end{rem}

\subsection{Pro-p rigidity for Demuskin type relations}

\subsubsection{p-efficient splittings}

We now describe a property which will allow us to move from showing
properties of a discrete group $\mathcal{G}$ to a Demuskin group
$G$.
\begin{defn}
A graph of discrete groups $\mathscr{G}=(X,F)$ is p-efficient if
$\pi_{1}(\mathscr{G})$ is residually p, each group $F_{x}$ is closed
in the pro-p topology on $\pi_{1}(\mathscr{G})$, and $\pi_{1}(\mathscr{G})$
induces the full pro-p topology on each $F_{x}$. 
\end{defn}

\begin{rem}
Let $(X,F)$ be a graph of discrete groups, e.g. if $X$ is one edge
this is either $F_{1}*_{L}F_{2}$ or $F_{1}*_{L}$. If it is p-efficient,
then we would get that $(X,\hat{F})$ is an injective graph (i.e.
the vertex groups inject to the fundamental group) of pro-p groups
and the pro-p completion of $\pi_{1}(X,F)$ is $\pi_{1}(X,\hat{F})$,
where by $\hat{}$ we mean pro-p completions, and the second $\pi_{1}$
is the fundamental group in the category of graphs of pro-p groups.

In the case of one edge graphs this would mean $\widehat{F_{1}*_{L}F_{2}}\cong\hat{F}_{1}*_{\hat{L}}\hat{F}_{2}$
and $\widehat{F_{1}*_{L}}\cong\hat{F}_{1}*_{\hat{L}}$, where the
amalgamated free product/HNN extension, on the left side of each isomorphism,
is in the category of discrete groups, and the right side is in the
category of pro-p groups.

In other words, if a graph of groups is p-efficient then taking fundamental
groups commutes taking pro-p completions, and the discrete fundamental
group injects into the pro-p one.
\end{rem}

\begin{prop}
Every group $\mathcal{G}_{r'}$ is residually p.
\end{prop}

\begin{proof}
by \cite{gildenhuys1975one}.
\end{proof}
\begin{prop}
\label{prop:p torsion }Let $\alpha\in Split_{D}(\mathcal{G}_{r'},\mathbb{Z})$.
Then for any $s$, there is a p-group quotient of $\mathcal{G}$ such
that the image of $c$, a generator of $\mathcal{G}_{e}$, is $p^{s}$
torsion. In particular, $\mathcal{G}$ induces the full pro p topology
on $\mathcal{G}_{e}$.
\end{prop}

\begin{proof}
If $\alpha$ is an HNN extension, then $c$ is a non trivial primitive
element in $\mathcal{G}^{ab}$ and so any map $\mathcal{G\twoheadrightarrow\mathbb{Z}}/p^{s}$,
sending $c$ to $1$ will do.

Otherwise, we have a basis for which $\alpha=F_{2n}*_{<c>}F_{2d-2n}$
where $c$ maps to $x_{1}^{q}[x_{1},y_{1}]...[x_{n},y_{n}]$ in $F_{2n}$
and to $[x_{n+1},y_{n+1}]...[x_{d},y_{d}]y_{d}^{-p^{r'}}$ in $F_{2d-2n}$.

If $r',r\geq s$, we define maps from $F_{2n}$, $F_{2d-2n}$ to the
Heisenberg group with coefficient in $\mathbb{Z}/p^{s}$
\[
\mathcal{H}(\mathbb{Z}/p^{r})=\{\left(\begin{array}{ccc}
1 & x & y\\
0 & 1 & z\\
0 & 0 & 1
\end{array}\right):x,y,z\in\mathbb{Z}/p^{s}\}
\]
 sending $x_{n+1},x_{1}$ to $\left(\begin{array}{ccc}
1 & 1 & 0\\
0 & 1 & 0\\
0 & 0 & 1
\end{array}\right)$, $y_{n+1},y_{1}$ to $\left(\begin{array}{ccc}
1 & 0 & 0\\
0 & 1 & 1\\
0 & 0 & 1
\end{array}\right)$ and the rest of $x_{i},y_{i}$ to the identity. Thus we have the
image of $c$ will be which is of order exactly $p^{s}$. The map
$\mathcal{G}\rightarrow\mathcal{H}(\mathbb{Z}/p^{s})*_{\mathbb{Z}/p^{s}}\mathcal{H}(\mathbb{Z}/p^{s})\rightarrow\mathcal{H}(\mathbb{Z}/p^{s})$
where the final map identifies the two copies of $\mathcal{H}(\mathbb{Z}/p^{s})$
is the required quotient.

If $r<s\leq r'$, then we send $F_{2n}$ to $\mathbb{Z}/p^{r+s}$,
sending all generators to the identity, hence $c$ to an element of
order $p^{s}$, and we send $F_{2d-2n}$ to $\mathcal{H}(\mathbb{Z}/p^{r+s})$,
but now send $x_{n+1}$ to $\left(\begin{array}{ccc}
1 & p^{r} & 0\\
0 & 1 & 0\\
0 & 0 & 1
\end{array}\right)$.

The image of the map $\mathcal{G}\rightarrow\mathcal{H}(\mathbb{Z}/p^{r+s})*_{\mathbb{Z}/p^{r+s}}\mathbb{Z}/p^{r+s}=\mathcal{H}(\mathbb{Z}/p^{r+s})$
is the required quotient.

If $r'<s$ then we send $F_{2n}$ and $F_{2d-2n}$ to $\mathbb{Z}/p^{r'+s}$,
sending all generators to the identity, except $x_{1}$ which is sent
to $p^{r'-r}$ hence $c$ to an element of order $p^{s}$, which gives
us the required quotient of $\mathcal{G}$.
\end{proof}
\begin{defn}
Let $P$ be a finite p-group. A chief series for $P$ is a sequence
$1\leq P_{n}\leq P_{n-1}\leq...\leq P_{2}\leq P_{1}=P$ of normal
subgroups of $P$ such that each quotient $P_{i}/P_{i+1}$ is either
trivial or isomorphic to $Z/pZ$. 
\end{defn}

We first need the following theorem of Higman \cite{higman1964amalgams}:
\begin{thm}
\label{thm: Higman} Let $A$ and $B$ be finite $p$-groups with
common subgroup $A\cap B=U$. Then $A*_{U}B$ is residually $p$ if
and only if there are chief series $\{A_{i}\}$ and $\{B_{i}\}$ of
$A$ and $B$ respectively such that $\{U\cap A_{i}\}=\{U\cap B_{i}\}$.
In particular, $A*_{U}B$ is residually $p$ when $U$ is cyclic. 
\end{thm}

\begin{prop}
\label{prop:Amalgam is efficient}A splitting in $Split_{D,amalg}(\mathcal{G}_{r'},\mathbb{Z})$,
$\mathcal{G}_{r'}=G_{1}*_{L}G_{2}$ is a $p$-efficient splitting. 
\end{prop}

\begin{proof}
Let $H\vartriangleleft_{p}G_{1}$ and $P_{1}=G_{1}/H$, and suppose
$LH/H\cong Z/p^{r}Z$. By Proposition \ref{prop:p torsion } there
is a $p$-group quotient $G_{2}\rightarrow P_{2}$ such that the image
of $L$ is also isomorphic to $Z/p^{r}Z$. The quotient $P_{1}\ast_{Z/p^{r}Z}P_{2}$
obtained is residually p (by theorem \ref{thm: Higman}), so it has
a $p$-group quotient $Q$ distinguishing all the elements of $P_{1}$.

The kernel of the composite map $G_{1}\rightarrow\mathcal{G}\rightarrow P_{1}*_{Z/p^{r}Z}P_{2}\rightarrow Q$
is $H$, so $\mathcal{G}$ induces the full pro-$p$ topology on $G_{1}$,
and similarly for $G_{2}$.

Given $g\in\mathcal{G}\backslash G_{1}$, we can write $g=a_{1}b_{2}\cdot\cdot\cdot a_{n}b_{n}$
where $a_{i}\in G_{1}$, $b_{i}\in G_{2}$ and all are not in $L$,
except possibly $b_{n}=1$ or $a_{1}\in L$.

Next, we have $H_{i}\vartriangleleft_{p}G_{i}$ such that the images
of every $a_{j}$ (resp. $b_{j}$) in $P_{1}=G_{1}/H_{1}$ is (resp.
$P_{2}=G_{2}/H_{2}$) not in the image of $L$.

By proposition \ref{prop:p torsion } we can actually take $H_{i}$
to be such that the image of L is $Z/p^{r}Z$ in both $P_{i}$.

By construction the image $\phi(g)$ is not in $P_{1}$, where $\phi:\mathcal{G}\rightarrow P_{1}*_{Z/p^{r}}P_{2}$
the quotient map.

Since $P_{1}*_{Z/p^{r}}P_{2}$ is residually $p$ (by theorem \ref{thm: Higman}),
we can find a $p$-group quotient $Q$ of $P_{1}*_{Z/p^{r}}P_{2}$
distinguishing $\phi(g)$ from $P_{1}$. 

The map $\mathcal{G}\rightarrow Q$ distinguishes $g$ from $G_{1}$
and so $G_{1}$ is $p$-separable in $\mathcal{G}$, and similarly
for $G_{2}$.
\end{proof}
We now need the following theorem of Chatzidakis \cite{chatzidakis1994some}:
\begin{thm}
\label{thm:Chatzidakis} Let $P$ be a finite $p$-group, let $A$
and $B$ be subgroups of $P$, and let $f:A\rightarrow B$ be an isomorphism. 

Suppose that $P$ has a chief series $\{P_{i}\}$ such that $f(A\cap P_{i})=B\cap P_{i}$
for all $i$ and the induced map $f_{i}:AP_{i}\cap P_{i-1}/P_{i}\rightarrow BP_{i}\cap P_{i-1}/P_{i}$
is the identity for all $i$.

Then $P$ embeds in a finite $p$-group $T$ in which $f$ is induced
by conjugation. Hence the $HNN$ extension $P*_{f}$ is residually
$p$. 
\end{thm}

\begin{prop}
\label{prop:HNN efficient}A splitting in $Split_{D,HNN}(\mathcal{G}_{r'},\mathbb{Z})$,
$\mathcal{G}_{r'}=H*_{\theta}$, for $\theta:A\rightarrow B$ the
isomorphism of cyclic groups, is a p-efficient splitting. 
\end{prop}

\begin{proof}
Let $x_{1},...,x_{d-1},y_{d-1},y_{d}$, be generators of $\mathcal{G}_{r'}$
such that $A=y_{d}$ and $B=y_{d}^{-p^{r'}}x_{1}^{p^{r}}[x_{1},y_{1}]...[x_{d-1},y_{d-1}]y_{d}$.

First we will prove that $\mathcal{G}_{r'}$ induces the full pro-p
topology on $H$.

Let $P=H/\gamma_{n}^{(p)}(H)$ be one of the lower central $p$-quotients
of $H$, and let $\phi:H\rightarrow P$ be the quotient map.

Let $a$ and $b$ denote generators of $\phi(A)$ and $\phi(B)$ respectively,
and 

Since commutator subgroups and terms of the lower central p-series
are verbal subgroups, there is a commuting diagram: 

\[
\xymatrix{ & H\ar[dl]\ar[dr]\\
P\ar[d] &  & H^{ab}\ar[d]\\
P^{ab}\ar[rr]\sp(0.37)\cong &  & H^{ab}/\gamma_{n}^{(p)}(H^{ab})
}
\]
and so $P^{ab}=H^{ab}\otimes Z/p^{n-1}Z$.

The image of $a$ in $P$ thus has order at least $p^{n-1}$, also
by definition of the lower central series any element of $P$ has
order at most $p^{n-1}$.

Hence the image of $A$ in $P$ injects into $P^{ab}$.

Since $P$ is a characteristic quotient of $H$ and there is an automorphism
of $H$ taking $a$ to $b$, the order of $b$ will also be $p^{n-1}$.

Furthermore, if we denote by $\bar{a},\bar{b}$ the image of $a$
and $b$ in $P^{ab}/<\psi(x_{1})>$, where $\psi$ is the map from
$H$ to $P^{ab}$. We have $\bar{a}^{-p^{r'}}\bar{a}=\bar{b}$.

Now construct a chief series $\{P_{i}\}$ for $P$ whose first $n$
terms are the preimages of the terms of a chief series for $P^{ab}/<\psi(x_{1})>$
which intersects to a chief series on the subgroup of $P^{ab}/<\psi(x_{1})>$
generated by the image of $a$.

Then for $i\geq n$, we have $\phi(A)\cap P_{i}=\phi(B)\cap P_{i}=1$
and for $i<n$ the conditions of theorem \ref{thm:Chatzidakis} hold
by construction.

Hence $P*_{\phi(A)}$ is residually $p$, and we may take a $p$-group
quotient $P*_{\phi(A)}\rightarrow Q$ in which no element of $P$
is killed.

The kernel of the composite map $H\rightarrow G\rightarrow P*_{\phi(A)}\rightarrow Q$
is $\gamma_{n}^{(p)}(H)$ as required.

To show that $H$ is $p$-separable in $G$, proceed as in the proof
of proposition \ref{prop:Amalgam is efficient}: write $g\in G\backslash H$
as a reduced word in the sense of $HNN$ extensions, and take a sufficiently
deep lower central $p$-quotient $P=H/\gamma_{n}^{(p)}(H)$ so that
the image of $g$ in $P*_{\phi(A)}$ is again a reduced word not in
$P$. 

As shown above, $P*_{\phi(A)}$ is residually $p$, so admits a $p$-group
quotient $Q$ distinguishing the image of $g$ from the image of $P$. 

This quotient $Q$ of $G$ exhibits that $H$ is $p$-separable in
$G$.
\end{proof}
We get from the above:
\begin{cor}
\label{cor:p-eff splits}Every $\mathbb{Z}$ splitting of $\mathcal{G}_{r'}$
in $Split_{D}(\mathcal{G}_{r'},\mathbb{Z})$ is p-efficient .
\end{cor}

\begin{rem}
As mentioned in the introduction, the proof of propositions \ref{prop:Amalgam is efficient}
and \ref{prop:HNN efficient} is a direct generalization of the proof
of Wilkes \cite{wilkes2017virtual}.
\end{rem}

\begin{lem}
\label{lem:cong formula}Let $\alpha$ be a splitting of $G$, and
let $\phi\in Aut(G)$ then: $T_{\phi(\alpha)}^{k}=\phi T_{\alpha}^{k}\phi^{-1}$.
\end{lem}

\begin{proof}
Let $\alpha$ be a splitting of $G$ into an amalgamated free product
$G_{1}*_{H}G_{2}$, from the universal property of amalgamated free
products, (see definition \ref{def:pro-p amalgamted}) we get that
it is enough to show $T_{\phi(\alpha)}^{k}=\phi T_{\alpha}^{k}\phi^{-1}$
on the two subgroups $\phi(G_{1}),\phi(G_{2})$. 

Since Dehn twists act trivially on the first factor, then we get $T_{\phi(\alpha)}^{k}=\phi T_{\alpha}^{k}\phi^{-1}$
on $\phi(G_{1})$.

On $\phi(G_{2})$, $T_{\phi(\alpha)}^{k}$ acts by conjugation by
$(\phi(h))^{k}$ where $h$ is the fixed generator of $H$.

Let $\phi(g)\in\phi(G_{2})$, then $\phi T_{\alpha}^{k}\phi^{-1}(\phi(g))=\phi(T_{\alpha}^{k}(g))=\phi(h^{-k}gh^{k})=(\phi(h))^{-k}\phi(g)(\phi(h))^{k}$,
hence again the action is by conjugation by $(\phi(h))^{k}$.

From this we get $T_{\phi(\alpha)}^{k}=\phi T_{\alpha}^{k}\phi^{-1}$
for $\alpha$ an amalgamated free product.

The case of $\alpha$ an HNN extension is very similar, using the
universal property of HNN extensions (see definition \ref{def: pro-p HNN}).

\end{proof}
\begin{cor}
\label{cor:orbits dehn}If $\alpha,\beta\in Split(G,\mathbb{Z}_{p})$
are in the same $Aut(G)$ orbit, then $T_{\alpha}^{k}$ is inner if
and only if $T_{\beta}^{k}$ is inner.
\end{cor}

\begin{proof}
Let $\phi\in Aut(G)$ be the automorphism for which $\phi(\alpha)=\beta$
and assume $T_{\alpha}^{k}$ is the automorphism given by conjugation
by $g\in G$.

Let $x\in G$, then $T_{\beta}^{k}(x)=T_{\phi(\alpha)}^{k}(x)=\phi T_{\alpha}^{k}\phi^{-1}(x)=\phi(g^{-1}\phi^{-1}(x)g)=\phi(g^{-1})\phi(\phi^{-1}(x))\phi(g)=\phi(g^{-1})x\phi(g)$,
and so we get $T_{\beta}^{k}$ is inner, given by conjugation by $\phi(g)$.

The other direction follows in a similar manner.
\end{proof}

\subsubsection{Complexes of cyclic splittings}

We now show that actually any splitting of $G$ ``comes from'' a
splitting of $\mathcal{G}_{r'}$ for some $r'\geq r$.
\begin{defn}
For splittings $x$ and $y$ we say $x\leq y$ if $x$ is a refinement
of $y$.
\end{defn}

\begin{thm}
\label{thm: equivariant CC map}For every $r'\geq r$ we have a $Aut(\mathcal{G}_{r'})$
-equivariant map $j:Split_{D}(\mathcal{G}_{r'},\mathbb{Z})\rightarrow Split(G,\mathbb{Z}_{p})$
with the further properties that $x\leq y$ if and only if $j(x)\leq j(y)$
and if $z\in Split(G,\mathbb{Z}_{p})$ has $z\leq j(y)$ then $z=j(x)$
for some $x\leq y$.
\end{thm}

\begin{proof}
Construction of $i$:

Take $\alpha\in Split_{D}(\mathcal{G}_{r'},\mathbb{Z})$ for some
$r'$, it is is p-efficient by corollary \ref{cor:p-eff splits}.

Therefore by \cite{ribes2017profinite} propositions 6.5.3 and 6.5.4
there is an induced splitting of $G$ as an injective graph of pro-p
groups with procyclic edge groups.

The standard tree dual to this splitting (see \cite{ribes2017profinite}
theorem 6.5.2) is a representative of an element $j(\alpha)$ of $Split(G,\mathbb{Z}_{p})$. 

Note that the Bass-Serre tree of the splitting $\alpha$ of $\mathcal{G}_{r'}$
is naturally embedded as an abstract subgraph of the Bass-Serre tree
of $j(\alpha)$ which is dense in the topology on it, by again \cite{ribes2017profinite}
propositions 6.5.4.

We have $x\leq y$ if and only if $j(x)\leq j(y)$, due to the fact
that $x\leq y$ if only if the graph of groups decomposition of $\mathcal{G}_{r'}$
corresponding to $x$ is obtained from the decomposition corresponding
to $y$ by collapsing some sub-graphs-of-groups; the same collapsed
subgraph can be found in $j(x)\leq j(y)$.

If $z\leq j(y)$ the graphs of groups of $z$ is obtained from $j(y)$
by collapsing some sub-graphs-of-groups; doing such collapses on the
graph of discrete groups corresponding to y, we get the required $x$. 
\end{proof}
\begin{rem}
The property on refinements gives us that the map is actually a map
between the complexes $\mathcal{C}_{D}(\mathcal{G}_{r'})$ and $\mathcal{C}(G)$.
\end{rem}

We now want:
\begin{thm}
\label{thm:curve complex bijection}The map 
\[
j:\amalg_{r'\geq r}Split_{D,amalg}(\mathcal{G}_{r'},\mathbb{Z})/Aut(\mathcal{G})\amalg Split_{D,HNN}(\mathcal{G},\mathbb{Z})/Aut(\mathcal{G})\rightarrow Split(G,\mathbb{Z}_{p})/Aut(G)
\]
 is a bijection which has $1$ orbit for HNN-extension, and exactly
1 orbit for each pair of integers $r'\geq r$ and $1\leq n\leq d$.
\end{thm}

\begin{proof}
Let $\alpha$ be either the $HNN$ extension from theorem \ref{thm:HNN discrete Demuskin}
or \ref{thm:amalgam discrete demuskin} for some integer $n$ between
1 and $d$.

We will show that any $\beta\in Split(G,\mathbb{Z}_{p})$ has some
$\psi\in Aut(G)$ such that $\psi(\beta)=j(\alpha)$.

Take such $\beta\in Split(G,\mathbb{Z}_{p})$. 

Suppose first that $\beta$ is an amalgamated pro-p product $G_{1}*_{L}G_{2}$. 

From the classification of $PD^{2}$ pairs in \cite{wilkes2020classification}
Theorem 3.3, and theorem 5.18 in \cite{wilkes2019relative}, we have
that there is some basis, such that $G_{1}$ is free of rank $2n$
and $G_{2}$ is free of rank $2d-2n$, and for which the image of
$L$ in $G_{1}$ is $x_{1}^{p^{s}}[x_{1},y_{1}]...[x_{n},y_{n}]$
and in $G_{2}$ is $y_{d}^{p^{s'}}[y_{d},x_{d}]...[y_{n+1},x_{n+1}]$,
for some $s$ and $s'$.

Since the dualizing module of the pairs , is the restriction of $G$,
we have $s,s'\geq r$.

On the other hand, if both $s,s'>r$ we would have that mod $p^{r}$
both maps are trivial, and so from the universal property of amalgamated
products, we would get that so is $\chi_{G}$, in contradiction to
the definition of $r$.

Thus either $s'=r$ or $s=r$. Suppose $s=r$.

Take $\alpha\in Split_{D,amalg}(\mathcal{G}_{r'},\mathbb{Z})\text{ to be }\text{\ensuremath{\mathcal{G}_{r'}}}=F_{2n}*_{L'}F_{2d-2n}$
for $r'=s'$ and for $2n=rk(G_{1})$ and $L'=<c>$ is cyclic and we
map $c$ to $x_{1}^{q}[x_{1},y_{1}]...[x_{n},y_{n}]$ in $F_{2n}$
and to $y_{d}^{p^{r'}}[y_{d},x_{d}]..[y_{n+1},x_{n+1}]$ in $F_{d-2n}$
for some basis $x_{i},y_{i}$.

Looking at $j(\alpha)$ we get isomorphisms from $(G_{1},L)$ to $j((F_{2n},L'))$
and $(G_{2},L)$ to $j((F_{2d-2n},L'))$ as pairs of groups (see \cite{wilkes2019relative})
, which glue together to an automorphism of $G$ sending $j(\alpha)$
to $\beta$.

Given another splitting $\beta'$ of $G=G'_{1}*_{L''}G'_{2}$ for
which there is $\varphi\in Aut(G)$ for which $\varphi(\beta')=\beta$,
we must have (after possibly changing between $G'_{1}$ and $G'_{2}$)
that $G'_{1}$ is free of rank $2n$ and $G'_{2}$ is free of rank
$2d-2n$, and for which the image of $L$ in $G_{1}$ is $x_{1}^{p^{s}}[x_{1},y_{1}]...[x_{n},y_{n}]$
and in $G_{2}$ is $y_{d}^{p^{s'}}[y_{d},x_{d}]...[y_{n+1},x_{n+1}]$,
for the same $s$ and $s'$ of $G_{1},G_{2}$ .

For the case of $\beta$ is an HNN extension $G_{1}*_{L}$, one needs
to be a bit more careful in finding two isomorphisms to $F_{2d-1}$
such that they will glue properly to an automorphism of $G$.

Such gluing works, the proof for which is the same as the end of part
(3) of theorem 4.6 of \cite{wilkes2020classification}.
\end{proof}

\begin{rem}
As mentioned in the introduction the above theorem, gives us that
every splitting comes from some known list of splittings. This, and
the proof of it, is similar to the case of pro-p completions of surface
groups as seen in \cite{wilkes2020classification}.

Due to the $\mathcal{G}_{r'}$ having isomorphic pro-p completions,
but having infinitely many isomorphism classes as discrete group (see
lemma \ref{lem:non isomorphic}), we need to take them all. On the
other hand, since we want a bijection, and HNN extensions all have
only a single orbit, we need to take HNN extensions only from one
$\mathcal{G}_{r'}$, we choose $r'=\infty$ for convenience.
\end{rem}

We now need two final lemma about splittings of $G$:
\begin{lem}
\label{lem:p-cong sep}Let $\alpha\in Split_{D,amalg}(\mathcal{G}_{r'},\mathbb{Z})$,
and $\beta$ a splitting having non trivial intersection with $\alpha$.
Let $c$ be the generator of the edge of $\beta$, then there is a
homomorphism $\varphi:G\rightarrow P$ such that $P$ is a finite
p-group, and $\varphi(c)$ is not conjugate to $\varphi(T_{\alpha}^{k}(c))$.
\end{lem}

\begin{proof}
The lemma is the same as proposition 4.5 in \cite{paris2009residual}
(there it was proven for surfaces), where $g=c$, $h=\varphi(T_{\alpha}^{k}(c))$
and the splitting is $\alpha$. 

Since we are looking to separate an element from its images under
the Dehn twist of $\alpha$, we see that we are exactly in case 3
of the proof of proposition 4.5 in \cite{paris2009residual}, and
for such a case, the exact same proof follows for our groups $\mathcal{G}_{r'}$.
\end{proof}
\begin{lem}
\label{lem:p-cong sep HNN}Let $\alpha\in Split_{D,HNN}(\mathcal{G},\mathbb{Z})$,
then there is a splitting $\beta$ having non trivial intersection
with $\alpha$, such that there is a homomorphism $\varphi:G\rightarrow P$
such that $P$ is a finite p-group, and $\varphi(c)$ is not conjugate
to $\varphi(T_{\alpha}^{k}(c))$, where $c$ is the generator of the
edge of $\beta$.
\end{lem}

\begin{proof}
After change of coordinates, we can assume $\alpha$ is the splitting
of $\mathcal{G}$ into $A*_{\theta}$ where $A$ be the free group
on $d-1$ generators $x_{1},y_{1}...,x_{d}$, and let $\theta$ be
the map sending $x_{d}$ to $x_{1}^{p}[x_{1},y_{1}]...[x_{d-1},y_{d-1}]x_{d}$.

Let $\beta$ be the splitting $B=<x_{1},y_{1},..,x_{d-1},y_{d},b>$
where $b=y_{d-1}^{-1}x_{d}$ and $\theta'(b^{-1}x_{d-1}^{-1}by_{d})=x_{d-1}^{-1}[y_{d-2},x_{n-2}]...[y_{1},x_{1}]x_{1}^{q}y_{d}$
is the gluing map we have: 

\[
B*_{\theta'}=\{B,t|\begin{array}{c}
tb^{-1}x_{d-1}^{-1}by_{d}t^{-1}=\\
x_{d-1}^{-1}[y_{d-2},x_{n-2}]...[y_{1},x_{1}]x_{1}^{q}y_{d}
\end{array}\}
\]
 where $t$ would be $x_{d}$ in $\mathcal{G}$.

Note that for the same reason as in theorem \ref{thm:HNN discrete Demuskin}
we would get $T_{\alpha}^{k}(\beta)\neq\beta$.

Let $c=x_{d}^{-1}y_{d-1}x_{d-1}^{-1}y_{d-1}^{-1}x_{d}y_{d}$.

Then the statement of our lemma is the same as proposition 4.5 in
\cite{paris2009residual} (there it was proven for surfaces), for
$g=c$, $h=\varphi(T_{\alpha}^{k}(c))$ and the splitting is $F_{2d-2}*_{x_{1}^{q}[x_{1},y_{1}]...[x_{d-1},y_{d-1}]=[y_{d},x_{d}]}F_{2}=\gamma$. 

Writing $c$ under the splitting $F_{2d-2}*_{x_{1}^{q}[x_{1},y_{1}]...[x_{d-1},y_{d-1}]=[y_{d},x_{d}]}F_{2}$,
we get $c=(x_{d}^{-1})(y_{d-1}x_{d-1}^{-1}y_{d-1}^{-1})(x_{d}y_{d})$
and $\varphi(T_{\alpha}^{k}(c))=(x_{d}^{-1})(y_{d-1}x_{d-1}^{-1}y_{d-1}^{-1})(x_{d}x_{d}^{k}y_{d})$.

We see that we are either in case 2 of proposition 4.5 in \cite{paris2009residual},
and since on $F_{2}$ the edge is $\gamma=[y_{d},x_{d}]$ we have
lemma 4.8, and so the same rest of the proof works, or in case 3 of
the proof of proposition 4.5 in \cite{paris2009residual}, and again
for such a case, the exact same proof follows for our group $\mathcal{G}$.

Note that the above almost corresponds to the following picture (again
only almost since $\beta$ got put in a presentation which is easier
to work with):

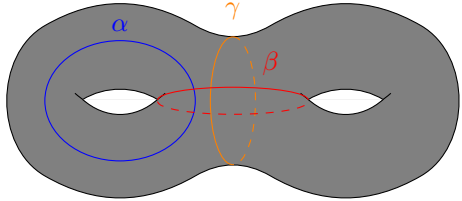
\begin{center} 
\begin{tikzpicture}
\filldraw[fill=gray] (0,1) to[out=30,in=150] (2,1) to[out=-30,in=210] (3,1) to[out=30,in=150] (5,1) to[out=-30,in=30] (5,-1) to[out=210,in=-30] (3,-1) to[out=150,in=30] (2,-1) to[out=210,in=-30] (0,-1) to[out=150,in=-150] (0,1);
\draw[smooth] (0.4,0.1) .. controls (0.8,-0.25) and (1.2,-0.25) .. (1.6,0.1);
\filldraw[fill=white][smooth] (0.5,0.02) .. controls (0.8,-0.25) and (1.2,-0.25) .. (1.5,0.02);
\filldraw[fill=white][smooth] (0.5,0.015) .. controls (0.8,0.2) and (1.2,0.2) .. (1.5,0.015);
\draw[smooth] (3.4,0.1) .. controls (3.8,-0.25) and (4.2,-0.25) .. (4.6,0.1);
\filldraw[fill=white][smooth] (3.5,0.02) .. controls (3.8,-0.25) and (4.2,-0.25) .. (4.5,0.02); 
\filldraw[fill=white][smooth] (3.5,0.015) .. controls (3.8,0.2) and (4.2,0.2) .. (4.5,0.015); 
\draw [color=blue](1,0) circle (1 and 0.8);
\draw [color=orange](2.5,0.85) arc(270:90:0.3 and -0.85);
\draw[color=orange][dashed] (2.5,0.85) arc(270:450:0.3 and -0.85);
\draw [color=red](1.5,0.01) arc(0:180: -1 and 0.17);
\draw [color=red][dashed](1.5,-0.01) arc(0:-180: -1 and 0.17);
\node at (1,1) {$ \color{blue} \alpha $}; 
\node at (3,0.5) {$ \color{red} \beta $}; 
\node at (2.5,1.2) {$ \color{orange} \gamma $}; 
\end{tikzpicture} 
\end{center}
\end{proof}
\begin{lem}
\label{lem:faithful action on pro-p}The action of $Out_{D}(\mathcal{G}_{r'})$
on $Split(G,\mathbb{Z}_{p})$ is faithful.
\end{lem}

\begin{proof}
First, by the proof of theorem \ref{thm:amalgam discrete demuskin}
and lemma \ref{lem:p-cong sep HNN}, we get that $T_{\alpha}^{k}$
has a splitting $\beta\in Split(\mathcal{G},\mathbb{Z})$ such that
$T_{\alpha}^{k}(\beta)\neq\beta$. We would like to show that for
that $\beta$ we also have $T_{\alpha}^{k}(j(\beta))\neq j(\beta)$. 

Due to $j$ being $Aut(\mathcal{G})$ equivariant, we get that we
want to show $j(T_{\alpha}^{k}(\beta))\neq j(\beta)$. 

Let $c$ be the generator of the edge of $\beta$, we would like to
show that $c$ and $T_{\alpha}^{k}(c)$ are not conjugate in $G$.

We already know from $T_{\alpha}^{k}(\beta)\neq\beta$ that they are
not conjugate in $\mathcal{G}$. 

By lemma \ref{lem:p-cong sep} or \ref{lem:p-cong sep HNN} we get
that $T_{\alpha}^{k}(c)$ and $c$ are not conjugate in $G$ and we
are done.
\end{proof}
Theorem \ref{thm: equivariant CC map} and lemma \ref{lem:faithful action on pro-p}
gives us:
\begin{cor}
\label{cor:injection of Out}The natural map $\phi:Out_{Dehn}(\mathcal{G})\rightarrow Out(G)$
is an injection.
\end{cor}

Since all splittings are $p$-efficient, we get: 
\begin{cor}
\label{cor:dis to pro-p dehn}We have that $\phi(T_{x}^{k})=T_{j(x)}^{k}$,
and so if $T_{x}^{k}$ is outer, then so is $T_{j(x)}^{k}$.
\end{cor}

\begin{cor}
\label{cor:main dehn outer}Given $\alpha\in Split(G,\mathbb{Z})$,
one has that $T_{\alpha}^{k}\neq0$ in $Out(G),$ for any $k\neq0$.
\end{cor}

\begin{proof}
Let $\alpha\in Split(G,\mathbb{Z}_{p})$, by theorem \ref{thm:curve complex bijection}
we have $\phi\in Aut(G)$ such that $\phi(\alpha)=j(x)$ for some
$x\in Split_{D}(\mathcal{G}_{r'},\mathbb{Z})$ for some $r'$.

By theorems \ref{thm:HNN discrete Demuskin} and \ref{thm:amalgam discrete demuskin},
$T_{x}^{k}$ is outer, and so by lemma \ref{cor:dis to pro-p dehn},
so is $T_{j(x)}^{k}$.

By corollary \ref{cor:orbits dehn} we get that $T_{\alpha}^{k}$
is outer, as required.
\end{proof}

\subsubsection{A family of non pro-$l$ rigid groups}

We finish with an application to group theory, coming from the examples
discussed above.

As we have seen, all the group $\mathcal{G}_{r'}$ have the same (non-free)
pro-p completion, namely the Demuskin group $G$.

On the other hand, for all $l\neq p$ we also have that they have
the same pro-$l$ completion, namely the free pro-l group on $2d-1$
generators.

This is due to the relation $w_{r'}:=x_{1}^{q}[x_{1},y_{1}]...[x_{d},y_{d}]y_{d}^{-p^{r'}}$
not being in $[\hat{F}_{l},\hat{F}_{l}](\hat{F}_{l})^{l}$, where
$\hat{F}_{l}$ is the free pro-$l$ group on generators $\{x_{1},y_{1},...,x_{d},y_{d}\}$,
and so $w_{r'}$ is primitive in $\hat{F}_{l}$.

We would like to show that these groups are non-isomorphic as discrete
groups, which we do in the following lemma:
\begin{lem}
\label{lem:non isomorphic}Given a group $\mathcal{G}_{t}$ for some
$t\geq r$, there are finitely many $k\geq r$ such that $\mathcal{G}_{t}\cong\mathcal{G}_{k}$.
In particular there are infinitely many isomorphism classes in the
family $\{\mathcal{G}_{r'}\}_{r'\geq r}$.
\end{lem}

\begin{proof}
By \cite{rosenberger1994isomorphism} theorem 3.4, we have that it
is enough to show that there are only finitely many Nielsen equivalence
classes in the family of words $w_{r'}:=x_{1}^{p^{r}}[x_{1},y_{1}]...[x_{d},y_{d}]y_{d}^{-p^{r'}}$
.

But by the whitehead algorithm \cite{whitehead1936equivalent}, we
see immediately that they are all already in minimal length, on the
other hand they have different lengths for each $r'$, and the lemma
follows.
\end{proof}
\begin{cor}
\label{cor:non rigidity}There is an infinite family of non-isomorphic
discrete groups, having the same pro-$l$ completion for all primes
$l$. 
\end{cor}

\newpage

\bibliographystyle{plain}

\end{document}